\documentclass[a4paper,11pt]{article}
\usepackage[utf8]{inputenc}
\usepackage[margin=1in]{geometry}
\usepackage[english]{babel}
\usepackage{graphicx}
\usepackage{fancyhdr}
\usepackage{geometry}
\usepackage{amsfonts,amssymb,amsthm,amsmath}
\usepackage{mathtools}
\usepackage{enumitem}
\usepackage[percent]{overpic}
\usepackage{tikz}
\usepackage{mathrsfs}
\usepackage{ifthen,tabularx,graphicx,multirow}
\usepackage{xargs}
\usepackage[colorinlistoftodos,prependcaption,textsize=tiny]{todonotes}
\usepackage{tcolorbox}
\usepackage{enumerate}
\usepackage{comment}
\usepackage{algorithm}
\usepackage{hyperref}
\usepackage{algpseudocode}
\usepackage{booktabs}
\usepackage{varwidth}
\usepackage[capitalise]{cleveref}
\usepackage{cite}

\definecolor{rev1}{HTML}{cb270f}
\definecolor{rev2}{HTML}{1c8235}

\newcommandx{\at}[2][1=]{\todo[linecolor=red,backgroundcolor=red!25,bordercolor=red,#1]{#2}}
\numberwithin{equation}{section}
\newtheorem{theorem}{Theorem}
\newtheorem{remark}{Remark}
\newtheorem{example}{Example}
\newtheorem{lemma}{Lemma}

\newtheorem{proposition}{Proposition}
\newtheorem{definition}{Definition}

\numberwithin{theorem}{section}
\numberwithin{lemma}{section}
\numberwithin{corollary}{section}
\numberwithin{proposition}{section}
\numberwithin{definition}{section}
\numberwithin{example}{section}
\makeatletter 
\newcounter{algorithmicH}
\let\oldalgorithmic\algorithmic
\renewcommand{\algorithmic}{
  \stepcounter{algorithmicH}
  \oldalgorithmic}
\renewcommand{\theHALG@line}{ALG@line.\thealgorithmicH.\arabic{ALG@line}}
\makeatother

\newcommand*{\OS}{{\Omega_{\mathrm{SA}}}}

\newcommand*{\MAW}{{\mathcal{M}_{\mathrm{AW}}^U}}
\newcommand*{\spec}{{\mathrm{Sp}}}

\newcommand*{\dom}{{\mathcal{D}}}
\newcommand*{\ran}{{\mathrm{ran}}}
\newcommand*{\cl}[1]{{\text{\rm Cl}\left(#1\right )}}
\newcommand*{\dist}{{\mathrm{dist}}}
\newcommand*{\diag}{{\mathrm{diag}}} % use () type brackets

\newcommand*{\dAW}{{d_{\mathrm{AW}}^U}}
\newcommand*{\OD}{{\Omega_{\mathrm{D}}}}
\newcommand{\pderiv}[2]{\frac{\partial #1}{\partial #2}}

\providecommand{\keywords}[1]{\textbf{\textit{Keywords:}} #1}

\title{Universal Methods for Nonlinear Spectral Problems}
\author{Matthew J. Colbrook\thanks{DAMTP, University of Cambridge. ({m.colbrook@damtp.cam.ac.uk})} \and Catherine Drysdale\thanks{SMQB, University of Birmingham. (c.n.d.drysdale@bham.ac.uk)}}
\date{}

\begin{document}

\maketitle

\begin{abstract}
Nonlinear spectral problems arise across a range of fields, including mechanical vibrations, fluid-solid interactions, and photonic crystals. Discretizing infinite-dimensional nonlinear spectral problems often introduces significant computational challenges, particularly spectral pollution and invisibility, which can distort or obscure the true underlying spectrum. We present the first general, convergent computational method for computing the spectra and pseudospectra of nonlinear spectral problems. Our approach uses new results on nonlinear injection moduli and requires only minimal continuity assumptions: specifically, continuity with respect to the gap metric on operator graphs, making it applicable to a broad class of problems. We use the Solvability Complexity Index (SCI) hierarchy, which has recently been used to resolve the classical linear problem, to systematically classify the computational complexity of nonlinear spectral problems. Our results establish the optimality of the method and reveal that Hermiticity does not necessarily simplify the computational complexity of these nonlinear problems. Comprehensive examples -- including nonlinear shifts, Klein--Gordon equations, wave equations with acoustic boundary conditions, time-fractional beam equations, and biologically inspired delay differential equations -- demonstrate the robustness, accuracy, and broad applicability of our methodology.
\end{abstract}


\section{Introduction}

Nonlinear spectral problems in which the spectral parameter appears nonlinearly arise in numerous applications, including mechanical vibrations~\cite{lancaster2002lambda}, fluid-solid interactions~\cite{voss2002rational}, photonic crystals~\cite{sakoda2001photonic}, time-delay systems~\cite{jarlebring2008spectrum}, resonances~\cite{bindel2006theory}, and many other areas~\cite{steinbach2009boundary,kukelova2008polynomial,mehrmann2011nonlinear}. Such problems can originate from eigenvalue-dependent boundary conditions~\cite{botchev2009svd} or material parameters~\cite{engstrom2010complex}, to name a few. They can also arise as artifacts of discretization using particular basis functions~\cite{betcke2005reviving} or from truncating an infinite domain with transparent boundary conditions~\cite{liao2010nonlinear}. In the latter case, the truncation leads to a nonlinear eigenvalue problem.

Solving nonlinear spectral problems accurately is essential, as even small errors can lead to drastically incorrect predictions in mechanical systems, fluid dynamics, and photonic material design, potentially causing operational failures or misleading scientific insights. However, as recently demonstrated in~\cite{colbrook2025avoiding} using several benchmark problems, discretizing infinite-dimensional nonlinear spectral problems can introduce significant challenges. Standard techniques, such as computing eigenvalues of increasingly large finite-dimensional matrices by refining a mesh or increasing a truncation dimension, can fail to capture the true spectrum. This failure can take two forms: \textit{spectral invisibility}, where parts of the spectrum are missed, and \textit{spectral pollution}, where spurious eigenvalues persist in the approximation of the spectrum. 

A convergent method is one that avoids spectral invisibility and pollution. While~\cite{colbrook2025avoiding} introduced a convergent contour-based method for holomorphic problems with discrete spectrum, no general approach currently exists for nonlinear spectral problems. In this paper, we close this gap by presenting the first universally convergent computational framework for computing the spectra and pseudospectra of nonlinear spectral problems. Our methods apply in full generality: they require only that the nonlinear pencil is continuous in the gap topology and make no assumptions about whether the spectrum is discrete, continuous, or both. Finally, we prove that our methods are optimal by establishing fundamental impossibility results that hold in any model of computation.

\subsection{Classical (linear) computational spectral problem and SCI hierarchy}

``Can one compute the spectrum of a linear operator on a separable Hilbert space?" was an important open question for several decades that originated from the work of Szeg{\H o} \cite{Szego} on finite section approximations and Schwinger \cite{Schwinger} on finite-dimensional approximations to quantum systems.\footnote{See also the work of B{\"o}ttcher~\cite{Albrecht_Fields,Bottcher_pseu}, B{\"o}ttcher \& Silbermann~\cite{Bottcher_book,bottcher2006analysis}, Laptev \& Safarov~\cite{Laptev}, and Brown~\cite{brown2007quasi,Brown_2006,Brown_Memoars}. B{\"o}ttcher \& Silbermann~\cite{Albrecht1983} pioneered the combination of spectral computation and $C^*$-algebras.} This classical \textit{linear} problem was recently resolved using the Solvability Complexity Index (SCI) hierarchy~\cite{ben2015can,Hansen_JAMS}, which provides a structured way to understand how computationally difficult it is to solve certain problems, revealing intrinsic limitations that dictate the minimum complexity required for accurate numerical solutions.

A key insight from this work is that computing the spectrum of a bounded linear operator on a Hilbert space may require multiple successive limits. For example, the algorithm $\Gamma_{n_3,n_2,n_1}$ from \cite{Hansen_JAMS} converges to the spectrum in the Hausdorff metric via a three-level limiting process:
\begin{equation}
\label{SCI_mult_lim_example}
\lim_{n_3\rightarrow\infty}\lim_{n_2\rightarrow\infty}\lim_{n_1\rightarrow\infty}\Gamma_{n_3,n_2,n_1}(A)=\mathrm{Sp}(A)\quad\forall A\in\mathcal{B}(l^2(\mathbb{N})),
\end{equation}
where $\mathcal{B}(l^2(\mathbb{N}))$ denotes the class of bounded operators on $l^2(\mathbb{N})$. Later results confirmed that no algorithm can reduce the number of limits required~\cite{ben2015can}, revealing a fundamental analytical complexity in spectral computation for infinite-dimensional operators. However, for more restricted classes of operators, spectra can be computed with fewer limits~\cite{colbrook2019compute,colbrook3}.

In this paper, we provide results for nonlinear spectral problems. Building on the SCI framework~\cite{ben2015can,Hansen_JAMS,colbrook2020PhD}, we provide the first general methodology for computing spectra in this setting. The SCI hierarchy plays a central role by precisely classifying which nonlinear spectral problems are solvable and how many limits are required, thereby delineating the boundary between feasible and infeasible computation. This leads naturally to a classification theory and guides the construction of optimal algorithms. Despite its importance, the classification of spectral problems and the development of a comprehensive library of optimal algorithms remain largely unexplored.

%The difference between nonlinear spectral problems and linear spectral problems for the purposes of verified spectra is in the definition of the computation problem.

\subsection{Contributions and roadmap}

We consider nonlinear operator pencils of the form $T: U \to \mathcal{C}(\mathcal{H}_1, \mathcal{H}_2)$, where $\mathcal{C}(\mathcal{H}_1, \mathcal{H}_2)$ denotes the set of closed operators between two separable Hilbert spaces $\mathcal{H}_1$ and $\mathcal{H}_2$. Here, $U \subset \mathbb{C}$ is a domain, and we require only that $T$ is continuous with respect to the gap topology—that is, the topology induced by the gap metric on the graphs of operators. This continuity assumption is extremely weak and covers most applications of interest. It allows us to prove a perturbation result for injection moduli (\cref{prop_sigma_inf_cts}), which cannot hold under weaker assumptions such as strong continuity (\cref{exam_weaker_false}).

A detailed account of this framework is provided in \cref{sec:2:setup}, where we highlight key differences from the linear case and review the relevant aspects of the SCI hierarchy. We also construct an appropriate metric on the space of relatively closed subsets of $U$, with special care taken near the boundary $\partial U$. This generality allows us to treat singular examples, such as the one-dimensional pencil $T(z) = \sin(1/z)$ on $U = \mathbb{C} \setminus {0}$, and to compute spectra without encountering spectral pollution or invisibility.

In \cref{sec:convergent_algs}, we present our algorithms and prove their convergence. The core of our approach involves nonlinear injection moduli, which play a key role in enabling the computation of pseudospectra. We consider two sets of assumptions under which the algorithms operate: $\Lambda_1$, where we have access to matrix elements of $T$ with respect to suitable bases; and $\Lambda_2$, where we additionally have access to matrix elements of $T^*T$ and $TT^*$. Under assumption $\Lambda_1$, we show that pseudospectra can be computed through two successive limits and spectra through three successive limits. Under the stronger assumption $\Lambda_2$, the number of required limits is reduced to one for pseudospectra and two for spectra.

These methods provide the first general convergent framework for nonlinear spectral problems. A key advantage of the algorithms is that they converge monotonically—either from above or below—which makes them particularly well-suited for verification. For example, under assumption $\Lambda_2$, one can incorporate techniques such as interval arithmetic to turn the procedure into a fully rigorous computation. This ensures that the output sets are guaranteed subsets of the true pseudospectra, making the approach especially valuable in contexts such as computer-assisted proofs.

We also prove that our algorithms are optimal by establishing lower bounds within the SCI hierarchy (\cref{SCI_bounds1}). Furthermore, we extend these bounds to Hermitian problems (\cref{hermitian_still hard}) and show—perhaps surprisingly—that, unlike in the linear case, Hermiticity does not necessarily lead to a lower SCI classification. In other words, the problem does not become inherently easier simply because the operator is Hermitian.

Finally, in \cref{sec:examples}, we demonstrate the effectiveness of our methods through a range of computational examples, including nonlinear shifts, Klein--Gordon equations, wave equations with acoustic boundary conditions, time-fractional viscoelastic beam equations, and a diffusive Lotka--Volterra predator–prey system with delay. These examples highlight several further important points:
standard naive methods may fail to converge;
our framework applies equally well to differential operators, discrete systems, and lattice models;
it enables the computation of pseudoeigenfunctions and yields sharper pseudospectral estimates than those obtained via nonlinear numerical range techniques (a key advantage for contour methods used in solving evolution equations via inverse Laplace transforms); and the approach performs effectively on complex, real-world systems. Together, these comprehensive numerical examples illustrate both the theoretical strength and the practical utility of our framework, enabling robust and reliable predictions across a broad spectrum of applications in physics, engineering, and biology.

\subsection{Connections with previous work}

Before presenting our results, it is useful to connect them with previous work. We have already mentioned~\cite{colbrook2025avoiding}, which develops a contour method for computing eigenvalues of holomorphic spectral problems. To our knowledge, the only existing SCI classification for nonlinear spectral problems is that by R{\"o}sler and Tretter~\cite{rosler2022computing}, who consider Klein-Gordon spectra. See also the work of Ben-Artzi, Marletta, and R\"osler on computing scattering resonances via spectral covers \cite{Jonathan_res,benartzi2020computing}. Work by Zworski~\cite{Zworski1,Zworski2} can be interpreted within the SCI hierarchy framework. In particular, the approaches presented in~\cite{Zworski1} (scattering resonances) and~\cite{dyatlov2015stochastic} (Pollicott--Ruelle resonances) rely on representing resonances as limits of non-self-adjoint spectral problems, and hence the SCI hierarchy is inevitable; see also~\cite{ZworskiSjostrand}. Additionally, see Galkowski and Zworski~\cite{galkowski2022viscosity} for viscosity limits in connection with the internal waves problem.

One motivation for this paper is the discovery by Chandler--Wilde, Chonchaiya, and Lindner~\cite{chandler2024spectral} that banded operators on $l^2(\mathbb{Z})$ admit spectral covers under suitable conditions, enabling algorithms for computing spectra and rigorous spectral exclusions. Their results can be viewed as an extension of Gershgorin's theorem~\cite{gershgorin1931uber}, generalizing it to a family of inclusion sets for the spectra of bi-infinite tridiagonal matrices and relying on the injection modulus. The authors in~\cite{chandler2024spectral} assume access to all finite patches of a given size for the operator. We anticipate that it may be possible to combine their techniques with those presented in this paper.

The multiple-limit structure in \cref{SCI_mult_lim_example} appears in other areas of computational mathematics. An early instance is Smale's polynomial root-finding problem with rational maps \cite{smale_question}, involving several successive limits as demonstrated by McMullen \cite{McMullen1,mcmullen1988braiding} and Doyle \& McMullen~\cite{Doyle_McMullen}. These results can be expressed within the SCI hierarchy \cite{ben2015can}, generalizing Smale's seminal contributions on the foundations of scientific computing and the existence of algorithms \cite{smale1981fundamental,Smale_Acta_Numerica,BCSS}. The SCI framework is now used extensively to investigate foundational computational questions across diverse areas, including resonance problems \cite{Jonathan_res,benartzi2020computing}, inverse problems \cite{AdcockHansenBook}, optimization problems \cite{SCI_optimization}, foundations of AI \cite{colbrook2022difficulty}, and data-driven dynamics \cite{colbrook2024limits}. Additionally, Olver, Townsend, and Webb have established a fundamental framework for infinite-dimensional numerical linear algebra and computational approaches to infinite data structures, yielding significant theoretical insights as well as practical algorithms \cite{Olver_Townsend_Proceedings,Olver_SIAM_Rev,webb_thesis,webb2021spectra}.

\section{Nonlinear spectral problems}
\label{sec:2:setup}

In this section, we describe the setup of our problem and also collect some useful results that will be used later in our proofs.

\subsection{A brief guide to the SCI hierarchy}

The Solvability Complexity Index (SCI) hierarchy~\cite{colbrook2020PhD,colbrook3,Hansen_JAMS}
provides a framework for classifying the difficulty of computational problems and for proving the optimality of algorithms. We briefly introduce the SCI hierarchy to establish notation and provide context for our results. We begin with a formal definition of a computational problem, discuss the structure of the algorithms used to approximate solutions, and then explain how to classify the problem's difficulty.

\begin{definition}[Computational problem]
A computational problem is a quadruple 
$\{\Xi, \Omega, \mathcal{M},\Lambda\}$ comprising of the following elements:
\begin{itemize}[leftmargin=7mm]
\item $\Xi: \Omega \to \mathcal{M}$, the quantity of interest we seek to compute.\\
-- In our context, this might be a set (such as the spectrum ${\rm Sp}$).
\item $\Omega$, the class of inputs.\\
-- In our context, this might be a class of nonlinear operator pencils.
\item $\mathcal{M}\equiv (\mathcal{M},d)$, a metric space used to evaluate how well $\Xi$ has been approximated.\\
-- In our context, this will be a suitable metric on relatively closed subsets of an open set $U$.
\item $\Lambda$, the evaluation set.  
This set of complex-valued functions $\Omega$ describes the information about $A\in \Omega$ 
to which algorithms are given access and must be sufficiently rich: if $f(A)=f(B)$ for all $f\in\Lambda$, then $\Xi(A)=\Xi(B)$. 
\end{itemize}
\end{definition}

To approximate the solution to a computational problem $\{\Xi, \Omega, \mathcal{M},\Lambda\}$, we apply an \emph{algorithm} $\Gamma: \Omega \to \mathcal{M}$. In this context, an algorithm can only access any $A \in \Omega$ through a \emph{finite number} of evaluation functions $f \in \Lambda_\Gamma(A) \subset \Lambda$. The set $\Lambda_\Gamma(A)$ could be adaptively chosen on the fly by the algorithm. However, the essential property is that the result of the algorithm depends solely on this information.

\begin{definition}[General algorithm]\index{general algorithm!deterministic}
\label{CHAP2_def:Gen_alg_NEW}
A general algorithm for a computational problem $\{\Xi,\Omega,\mathcal{M},\Lambda\}$ is a map $\Gamma:\Omega\to \mathcal{M}$ with the following property. For any $A\in\Omega$, there exists a non-empty finite subset of evaluations $\Lambda_\Gamma(A) \subset\Lambda$ such that if $B\in\Omega$ with $f(A)=f(B)$ for every $f\in\Lambda_\Gamma(A)$, then $\Lambda_\Gamma(A)=\Lambda_\Gamma(B)$ and $\Gamma(A)=\Gamma(B)$.
\end{definition}

This is the most basic consistency requirement for any deterministic computational device. We will also talk of an \textit{arithmetic algorithm}, where each $\Gamma(A)$ can be computed using $\Lambda$ and finitely many arithmetic operations and comparisons. Specifying $\Lambda$ is extremely important; for instance, if $\Lambda$ contains arbitrary approximations of $\Xi$, the computational problem becomes trivial. Alternatively, if available, we often wish to incorporate structure into $\Lambda$ or $\Omega$.

\begin{example}
If $\Omega$ consists of bounded linear operators on $\ell^2(\mathbb{Z})$, we could design an algorithm $\Gamma$ such that, for some fixed value of $n > 0$, it accesses $f_{j,k}(A) = \langle A e_j, e_k \rangle$ only for $-n \leq j,k \leq n$ and outputs the eigenvalues of the resulting $(2n+1)\times(2n+1)$ matrix (the finite section method). This is a general algorithm. We can use an arithmetic algorithm to approximate these eigenvalues to any desired accuracy (treated as an additional input).
\end{example}

Most spectral problems for operators on infinite-dimensional spaces cannot be solved exactly using a single algorithm (i.e., finite data), or even a sequence of algorithms $\{\Gamma_n\}$ that converge to the solution as $n \rightarrow \infty$. Instead, one must consider families of algorithms that produce increasingly accurate approximations, with the solution obtained in the limit of their outputs. The complexity of a spectral problem is determined by \emph{how many limits} are required to solve it. This motivates the concept of ``towers'' of algorithms, whose height corresponds to the number of nested limits involved.

\vspace{2mm}

\noindent\textbf{Tower of algorithms:} A \emph{tower of algorithms} of \emph{height $k$} has the following structure. At the base level is a set of algorithms $\{\Gamma_{n_k,\ldots,n_2,n_1}\}$ for the problem $\{\Xi,\Omega,\mathcal{M},\Lambda\}$, indexed by $k$ variables. In typical situations, each of these indices corresponds to a different aspect of the approximation, such as the truncation dimension in the finite section method. However, for our impossibility results, they can be anything. These level 1 algorithms each work with finite data. Subsequent levels are obtained by taking successive limits in each of these indices. For any $A\in\Omega$, the tower is constructed as:
\begin{center}
\begin{tabular}{rrl}
level~1: & $\Gamma_{n_k, \ldots, n_2, n_1}(A)$ \!\!\!\!\!\!& \\[7pt]
level~2: & $\Gamma_{n_k, \ldots, n_2}(A)$ \!\!\!\!&\!\! $\displaystyle{} =\ \lim_{n_1\to\infty} \Gamma_{n_k,\ldots, n_2, n_1}(A)$ \\[7pt]
$\vdots$\ \ \ \ \ \  &  & $\ \,\vdots$ \\[7pt]
level~$k$: & $\Gamma_{n_k}(A)$\!\!\!\!& $\displaystyle{} =\ \lim_{n_{k-1}\to\infty} \Gamma_{n_k,n_{k-1}}(A)$ \\[12pt]
solution:&  $\Xi(A)$\!\!\!\!& $\displaystyle{} =\ \lim_{n_k \to\infty} \Gamma_{n_k}(A)$.
\end{tabular}
\end{center}
\noindent{}The SCI is the fewest number of limits needed to solve a computational problem.

In the following classifications, we need to clarify which model of computation the towers use at their base level $\Gamma_{n_k, \ldots, n_2, n_1}$, designated subsequently by a superscript $\alpha$: $\alpha=A$ designates an arithmetic algorithm\footnote{Throughout, when constructing upper bounds, we work in real number arithmetic, following the foundational model of Smale and coauthors~\cite{BCSS}. However, extending our results to other models—such as interval arithmetic~\cite{tucker2011validated} or Turing machines with inexact input~\cite{turing1937computable}—is straightforward, since the resulting errors can be made to vanish in the first limit. See also the discussion of algorithmic unification in~\cite[Section 2.3]{colbrook2020PhD}.}; $\alpha=G$ designates a general algorithm. We will use the $\alpha=A$ model when constructing practical algorithms and the $\alpha=G$ model when establishing lower bounds on the number of limits needed. There are two primary reasons for these differences:
\begin{itemize}[leftmargin=4mm]
	\item When we prove impossibility results for general algorithms, they immediately apply to any model of computation. Simultaneously, we construct arithmetic algorithms for upper bounds. Hence, we obtain the strongest theorems possible.
	\item The proofs of impossibility results attribute the non-computability to the inherent nature of the problems rather than to overly restrictive operational types. \end{itemize}

\vspace{2mm}

\noindent\textbf{SCI hierarchy:} We can now classify the SCI for a problem $\{\Xi,\Omega,\mathcal{M},\Lambda\}$.
\begin{itemize}[leftmargin=7mm]
\item If $\{\Xi,\Omega,\mathcal{M},\Lambda\}$ can be solved exactly in finitely many operations, we say its SCI is zero, and write $\{\Xi,\Omega,\mathcal{M},\Lambda\} \in \Delta_0^\alpha$.
\item If $\{\Xi, \Omega,\mathcal{M},\Lambda\}$ can be solved with a set of algorithms
$\{\Gamma_n\}$ such that for all $A\in\Omega$, $d(\Gamma_n(A),\Xi(A)) \le 2^{-n}$
(i.e., we can compute the answer in one limit with full control of the error),
we write $\{\Xi,\Omega,\mathcal{M},\Lambda\} \in \Delta_1^\alpha$. Here, $2^{-n}$ can be replaced by any positive sequence tending to $0$ without changing the classification.
\item If $\{\Xi, \Omega,\mathcal{M},\Lambda\}$ can be solved for all $A\in\Omega$
with a tower of algorithms of height $k \geq 1$ or less, we write 
$\{\Xi,\Omega,\mathcal{M},\Lambda\} \in \Delta_{k+1}^\alpha$.
\end{itemize}
When $(\mathcal{M},d)$ has more structure, it is useful to define notions of convergence from above or below. In this paper, we shall use the following.
Let $\mathcal{T}_\alpha$ be the collection of type $\alpha$ towers of algorithms and $(\mathcal{M},d)$ be the Hausdorff or Attouch--Wets metric (defined below) induced by another metric space $(\mathcal{M}',d')$. We set $\Sigma^{\alpha}_0 = \Pi^{\alpha}_0 = \Delta^{\alpha}_0$ and for $m \in \mathbb{N}$ define
\begin{equation*}
\begin{split}
\Sigma^{\alpha}_{m}= \Big\{&\{\Xi,\Omega,\mathcal{M},\Lambda\} {\in} \Delta_{m+1}^\alpha  :  \exists \{\Gamma_{n_{m},\ldots,n_1}\} {\in} \mathcal{T}_\alpha,  \{X_{n_{m}}(A)\}{\subset}\mathcal{M} \text{ s.t. }\forall A \in\Omega\\
&\Gamma_{n_{m}}(A)  \mathop{\subset}_{\mathcal{M}'} X_{n_{m}}(A),\lim_{n_{m}\rightarrow\infty}\Gamma_{n_{m}}(A)=\Xi(A), d(X_{n_{m}}(A),\Xi(A))\leq 2^{-n_{m}}\Big\},\\[1ex]
\Pi^{\alpha}_{m} = \Big\{&\{\Xi,\Omega,\mathcal{M},\Lambda\} {\in} \Delta_{m+1}^\alpha  :  \exists \{\Gamma_{n_{m},\ldots,n_1}\} {\in} \mathcal{T}_\alpha,  \{X_{n_{m}}(A)\}{\subset}\mathcal{M} \text{ s.t. }\forall A \in \Omega\\
&\Xi(A)  \mathop{\subset}_{\mathcal{M}'} X_{n_{m}}(A),\lim_{n_{m}\rightarrow\infty}\Gamma_{n_{m}}(A)=\Xi(A),
 d(X_{n_{m}}(A),\Gamma_{n_{m}}(A))\leq 2^{-n_{m}} \Big\}.
\end{split}
\end{equation*}
Here, $\mathop{\subset}_{\mathcal{M}'}$ means inclusion in the metric space $(\mathcal{M}',d')$. For the Attouch--Wets metric space appended with the empty set, we require that $\Gamma_{n_{m}}(A)=\emptyset$ if $\Xi(A)=\emptyset$ for $\Sigma^{\alpha}_{m}$, and may let $X_{n_{m}}(A)=\Gamma_{n_{m}}(A)$ if $\Xi(A)=\emptyset$ for $\Pi^{\alpha}_{m}$.

\begin{remark}[Verification]
Even though problems in $\Sigma_1^\alpha \cup \Pi_1^\alpha \backslash \Delta_1^\alpha$ are considered non-computable, they can still be used in computer-assisted proofs. For example, consider a conjecture stating that there exists a spectral point in $\{\mathrm{Re}(z)>0\}$. If this is true and we have a $\Sigma_1^A$-tower, $\{\Gamma_n\}$, we can compute $\Gamma_n(A)$ successively and eventually verify the truth of this statement.
\end{remark}

\subsection{The input class and problem functions}

We now recall the notion of \textit{generalized convergence} for closed operators, which extends convergence in the operator norm topology to unbounded operators. This notion is defined via the gap between closed subspaces of a Hilbert space (or, more generally, a Banach space).

\begin{definition}
Let $M$ and $N$ be closed subspaces of a Hilbert space $\mathcal{H}$. We set
$$
\delta(M,N)=\sup_{u\in M,\|u\|=1}\mathrm{dist}(u,N),\quad \hat{\delta}(M,N)=\max\{\delta(M,N),\delta(N,M)\}.
$$
Here, $\delta(\{0\},N)=0$ and $\delta(M,\{0\})=1$ if $M\neq\{0\}$.
\end{definition}

\begin{remark}
Basic properties of this gap are provided in Kato \cite[Chapter IV.2.1]{kato2013perturbation}. A slight modification of this gap yields a metric first introduced by Krein and Krasnoselski \cite{krein1947fundamental}; see also Gohberg and Krein \cite{gokhberg1957basic}, Kato \cite{kato1958perturbation}, and Cordes and Labrousse \cite{cordes1963invariance}. Additionally, we refer the reader to the article by B\"ottcher and Spitkovsky, which discusses other notions of distances between subspaces \cite{bottcher2010gentle}.
\end{remark}

Recall that the graph of an operator $A:\mathcal{D}(A)\rightarrow\mathcal{H}_2$ (where $\mathcal{D}(A)$ is a subspace of $\mathcal{H}_1$) is $\mathrm{gr}(A)=\{(x,Ax):x\in\mathcal{D}(A)\}$. This graph is a subspace of the direct sum $\mathcal{H}_1\oplus \mathcal{H}_2$, and is closed precisely when $A$ is a closed operator. Hence, given closed operators $S,T\in\mathcal{C}(\mathcal{H}_1,\mathcal{H}_2)$, we define their gap as
$$
\delta(S,T)=\delta(\mathrm{gr}(S),\mathrm{gr}(T)),\quad \hat{\delta}(S,T)=\max\{\delta(S,T),\delta(T,S)\}.
$$
This gap possesses several nice properties. For example, if $S$ and $T$ are densely defined, then $\hat{\delta}(S^*,T^*)=\hat{\delta}(S,T)$. Given a sequence $\{T_n\}\subset\mathcal{C}(\mathcal{H}_1,\mathcal{H}_2)$ and $T\in \mathcal{C}(\mathcal{H}_1,\mathcal{H}_2)$, we say $T_n$ converges to $T$ in the generalized sense if $\lim_{n\rightarrow\infty}\hat{\delta}(T_n,T)=0$. The following well-known theorem characterizes this type of convergence.

\begin{theorem}[Basic properties of generalized convergence {\cite[Theorem IV.2.23]{kato2013perturbation}}]
\label{gap_properties}
Let $\{T_n\}\subset\mathcal{C}(\mathcal{H}_1,\mathcal{H}_2)$ and $T\in \mathcal{C}(\mathcal{H}_1,\mathcal{H}_2)$.
\begin{itemize}
	\item[\rm{(i)}] If $T$ is \textbf{bounded}: $\lim_{n\rightarrow\infty}\hat{\delta}(T_n,T)=0$ if and only if $T_n$ is bounded for sufficiently large $n$ with $\lim_{n\rightarrow\infty}\|T_n-T\|=0$.
	\item[\rm{(ii)}] If $T$ is \textbf{boundedly invertible}: $\lim_{n\rightarrow\infty}\hat{\delta}(T_n,T)=0$ if and only if $T_n$ is boundedly invertible for sufficiently large $n$ with $\lim_{n\rightarrow\infty}\|T_n^{-1}-T^{-1}\|=0$.
	\item[\rm{(iii)}] If $T_n$ and $T$ are \textbf{densely defined}: $\lim_{n\rightarrow\infty}\hat{\delta}(T_n,T)=0$ if and only if $\lim_{n\rightarrow\infty}\hat{\delta}(T_n^*,T^*)=0$.
\end{itemize}
\end{theorem}

We are interested in spectral problems involving operator pencils $z\mapsto T(z)\in\mathcal{C}(\mathcal{H}_1,\mathcal{H}_2)$ that are continuous with respect to $z$. To perform computations, we also suppose that we have an orthonormal basis $\{e_n\}_{n\in\mathbb{N}}$ of $\mathcal{H}_1$ and orthonormal basis $\{\hat{e}_n\}_{n\in\mathbb{N}}$ of $\mathcal{H}_2$. We then define the following class of nonlinear operator pencils, for which we will develop algorithms to compute their spectra.

\begin{definition}[Continuous operator pencils]
Let $U\subset \mathbb{C}$ be a domain (i.e., a non-empty, connected open set). The class of nonlinear operator pencils $\Omega_{\mathrm{NL}}^U$ is the set of maps $T:U\mapsto \mathcal{C}(\mathcal{H}_1,\mathcal{H}_2)$ satisfying the following three properties:
\begin{itemize}
	\item $T$ is continuous when $\mathcal{C}(\mathcal{H}_1,\mathcal{H}_2)$ is equipped with the topology induced by the gap $\hat{\delta}$;
	\item for any $z\in U$, $T(z)$ is densely defined and $\mathrm{span}\{e_n:n\in\mathbb{N}\}$ forms a core of $T(z)$;
	\item for any $z\in U$, $\mathrm{span}\{\hat{e}_n:n\in\mathbb{N}\}$ forms a core of $T(z)^*$.
\end{itemize}
In words, we only assume that these pencils are continuous and that appropriate linear spans of basis vectors form cores of $T(z)$ and $T(z)^*$.
\end{definition}

\begin{remark}[Domain of the pencil]
The domain $\mathcal{D}(T(z))$ need not be constant.
\end{remark}

Here are two simple examples of this setup.

\begin{example}[Linear families]\label{linear_example}
Let $\mathcal{H}_1=\mathcal{H}_2$, $U=\mathbb{C}$, and fix any densely defined $A\in\mathcal{C}(\mathcal{H}_1,\mathcal{H}_2)$ such that $\{e_n\}_{n\in\mathbb{N}}$ forms a core of $A$ and $A^*$. Define $T(z)=A-zI$, then $T\in\Omega_{\mathrm{NL}}^{\mathbb{C}}$.
\end{example}

\begin{example}[Holomorphic families of type (A)]
Suppose that $U\subset \mathbb{C}$ is a domain and $T:U\mapsto \mathcal{C}(\mathcal{H}_1,\mathcal{H}_2)$ is such that $\mathcal{D}(T(z))$ is independent of $z\in U$ and $z\mapsto T(z)u$ is holomorphic for any fixed $u\in\mathcal{D}(T(z))$. Continuity of $T$ is discussed in \cite[page 366]{kato2013perturbation}. If $\mathrm{span}\{e_n:n\in\mathbb{N}\}$ forms a core of $T(z)$ and $\mathrm{span}\{\hat{e}_n:n\in\mathbb{N}\}$ forms a core of $T(z)^*$, then $T\in\Omega_{\mathrm{NL}}^U$.
\end{example}

We now define the set of eigenvalues and the spectrum of a nonlinear pencil. These definitions generalize the linear case in \cref{linear_example}.

\begin{definition}[Eigenvalues]
Let $T\in\Omega_{\mathrm{NL}}^U$. A scalar $\lambda\in U$ is an eigenvalue of $T$ if there exists a non-zero vector $u\in\mathcal{D}(T(\lambda))$ with $T(\lambda)u=0$.
\end{definition}

\begin{definition}[Spectrum]
Let $T\in\Omega_{\mathrm{NL}}^U$. The spectrum of $T$ is the set
$$
\mathrm{Sp}(T)=\{z\in U:T(z)\text{ is not boundedly invertible}\},
$$
and contains (but may strictly contain) the set of eigenvalues.
\end{definition}

\begin{remark}
The spectrum $\mathrm{Sp}(T)$ need not be a closed set of $\mathbb{C}$, but part (ii) of \cref{gap_properties} shows that it is always a relatively closed (possibly empty) subset of $U$.
\end{remark}

To compute $\mathrm{Sp}(T)$, we will compute pseudospectra, defined as follows.

\begin{definition}[Pseudospectra]
Let $T\in\Omega_{\mathrm{NL}}^U$ and $\epsilon>0$. The ($\epsilon$-)pseudospectrum of $T$ is
\begin{equation*}
\spec_{\epsilon}(T)=\mathrm{Cl}_U\left({\left\{z\in U:\left\|T(z)^{-1}\right\|^{-1}<\epsilon\right\}}\right),
\end{equation*}
where we define $\|T(z)^{-1}\|^{-1}=0$ if $z\in\mathrm{Sp}(T)$ and $\mathrm{Cl}_U$ takes the closure relative to $U$.
\end{definition}

The need to include the strict inequality and closure in the definition of $\spec_{\epsilon}(T)$ is discussed in \cite{shargorodsky2008level}, and essentially follows because $\|T(z)^{-1}\|$ can be constant on open subsets, even in the linear case. The following proposition (whose proof is entirely analogous to the linear case) shows that $\spec_{\epsilon}(T)$ can be written in terms of unions of spectra of perturbed pencils.

\begin{proposition}
Let $\epsilon>0$ and $T\in\Omega_{\mathrm{NL}}^U$. Define the following set of perturbations:
$$
\mathcal{C}_U(\epsilon)=\left\{E:U\rightarrow\mathcal{B}(\mathcal{H}_1,\mathcal{H}_2) \text{ continuous with } \sup_{z\in U}\|E(z)\|<\epsilon\right\}.
$$
Then the pseudospectrum $\spec_{\epsilon}(T)$ can be equivalently written as
$$
\spec_{\epsilon}(T)=\mathrm{Cl}_U\left(\bigcup_{E\in\mathcal{C}_U(\epsilon)}\spec(T+E)\right).
$$
\end{proposition}

\begin{proof}
It is enough to show that
$$
{\left\{z\in U:\left\|T(z)^{-1}\right\|^{-1}<\epsilon\right\}}=\bigcup_{E\in\mathcal{C}_U(\epsilon)}\spec(T+E).
$$
Suppose first that $z\notin\spec(T)$ and $\|T(z)^{-1}\|^{-1}<\epsilon$. Then, there exists a vector $v\in\mathcal{H}_2$ of unit norm with $\|T(z)^{-1}v\|_{\mathcal{H}_1}>\epsilon^{-1}$. Let $u=T(z)^{-1}v\in\mathcal{H}_1$ and define the operator $E(z)=E:\mathcal{H}_1\rightarrow\mathcal{H}_2$ by $Ex=-v\langle x, u\rangle_{\mathcal{H}_1}/\|u\|_{\mathcal{H}_1}^2$. Then, $\|E\|=1/\|u\|_{\mathcal{H}_1}<\epsilon$ and $[T(z)+E]u=0$ so $z\in \spec(T+E)$. For the reverse set inclusion, suppose for a contradiction that $z\in\spec(T+E)$ for some $E\in\mathcal{C}_U(\epsilon)$ but that $\|T(z)^{-1}\|^{-1}\geq\epsilon$. Then
$$
\|T(z)^{-1}E(z)\|\leq \|T(z)^{-1}\|\|E(z)\|\leq \|E(z)\|/\epsilon<1.
$$
Note that $T(z)+E(z)=T(z)(I+T(z)^{-1}E(z))$. Using a Neumann series, we have
$$
(I+T(z)^{-1}E(z))^{-1}=\sum_{j=0}^\infty (-1)^j[T(z)^{-1}E(z)]^j,
$$
which converges because $\|T(z)^{-1}E(z)\|<1$. Hence, since $T(z)$ is invertible, so too is the product $T(z)(I+T(z)^{-1}E(z))$. It follows that $z\notin\spec(T+E)$, which is a contradiction.
\end{proof}

\subsection{The metric space and meaning of convergence}\label{sec:metric_def}

To formulate our computational problem, we must specify the metric space $(\mathcal{M},d)$. In our case, we want to compute spectra and pseudospectra, so $\mathcal{M}$ is the set of relatively closed subsets of $U$. Hence, we must define an appropriate metric on $\mathcal{M}$. Convergence of a sequence of approximations $C_n\rightarrow \spec(T)$ with respect to this metric should capture two key features:
\begin{itemize}
	\item \textbf{Avoidance of spectral pollution:} Any accumulation point in $U$ of the sequence of sets $\{C_n\}$ (i.e., any $z\in U$ with $\liminf_{n\rightarrow\infty}\dist(z,C_n)=0$) must belong to $\spec(T)$;
	\item \textbf{Avoidance of spectral invisibility:} If $z\in \spec(T)$, then $\lim_{n\rightarrow\infty}\dist(z,C_n)=0$.
\end{itemize}
To ensure that convergence in $(\mathcal{M},d)$ meets these requirements, particular care is needed when handling the boundary $\partial U$, as illustrated by the following simple example.

\begin{example}
Suppose that $U=\{z\in\mathbb{C}:|z|<1\}$ and $T\in\Omega_{\mathrm{NL}}^U$ with $\spec(T)=\{0\}$. The approximation sets $C_n=\{0,1-1/n\}$ avoid spectral pollution and spectral invisibility. However, the Hausdorff distance between $C_n$ and $\{0\}$ converges to $1$ as $n\rightarrow\infty$.
\end{example}

The reader might initially consider strengthening the avoidance of spectral pollution by taking into account any accumulation points $z\in \cl{U}$. However, although we consider pencils that are continuous with respect to the gap $\hat{\delta}$, they need not be uniformly continuous on the whole of $U$. This flexibility allows us to treat examples such as the one-dimensional problem $T(z)=\sin(1/z)$ for $U=\mathbb{C}\backslash \{0\}$. This issue is also related to numerical approximation, as illustrated by the following example.

\begin{example}
Suppose that $\mathcal{H}_1=\mathcal{H}_2=\mathcal{H}$ and $A$ is a positive self-adjoint operator with compact resolvent (e.g., the reader may think of an appropriate elliptic differential operator). We set $T(z)=A-\frac{1}{z}I$ and $U=\mathbb{C}\backslash \{0\}$. Then
$$
\spec(T)=\{1/z:z\in\spec(A)\backslash\{0\}\},
$$
which has $0$ as an accumulation point. When numerically approximating the spectrum of $A$, or equivalently $\spec(T)$, the approximations are typically accurate for the `small' eigenvalues of $A$ but potentially inaccurate for the larger ones. It is desirable that on any compact set these eigenvalues become increasingly accurate as the discretization improves—in other words, the errors are ``pushed out'' to infinity. For the pencil $T$, this corresponds to the error converging uniformly to zero on compact subsets of $\mathbb{C}\backslash \{0\}$, thus excluding the boundary point $\{0\}$.
\end{example}

A straightforward way to achieve the local uniform convergence described above, thereby avoiding spectral pollution and invisibility, is to equip $U$ with a Riemannian metric.

\begin{definition}
Given a domain $U\subset \mathbb{C}\cong \mathbb{R}^2$ and a point $p\in U$, we define the inner product $g_p:T_pU\times T_pU\rightarrow\mathbb{R}$ by
$$
g_p(u,v)=\frac{\langle u,v \rangle}{\min\{1,[\dist(p,\partial U)]^2\}},
$$
where $\langle u,v \rangle$ denotes the canonical Euclidean metric. We denote the resulting Riemannian metric $\{g_p\}$ by $g_U$ and the induced metric on $U$ by $d_U$.
\end{definition}

\begin{remark}
If $\partial U=\emptyset$ (i.e., $U=\mathbb{C}$), then $g_U$ is the usual Euclidean metric.
\end{remark}

The metric $d_U$ induces the usual topology on $U$. However, the following proposition shows that this distance diverges as we approach the boundary of $U$.

\begin{proposition}\label{balls_separated}
For $x_0\in U$ and $r>0$, define the open ball
$$
S_{r}(x_0)=\{x\in U:d_U(x_0,x)<r\}.
$$
Then $\inf_{z\in S_r(x_0)}\dist(z,\partial U)>0$.
\end{proposition}

\begin{proof}
Suppose for a contradiction that this was false for some $x_0\in U$ and $r>0$. Then there exists $x_n\in S_{r}(x_0)$ with $\lim_{n\rightarrow\infty}\dist(x_n,\partial U)=0$. For $c>0$, define the sets
$$
V_c=\{x\in U:\dist(x,\partial U)\leq c\}.
$$
By selecting a subsequence if necessary, we may assume that $x_n\in V_{2^{-n}}$.We may also choose $N$ so that $x_0\not\in V_{2^{-N}}$. Let $n>N$ and consider any piecewise smooth path $h$ from $x_0$ to $x_n$ (note that $U$ is path connected since it is an open connected subset of $\mathbb{C}$). Let $h_m$ denote the part of this path that lies in $V_{2^{-m}}\backslash V_{2^{-(m+1)}}$ for $m=N,\ldots,n-1$. For any $x,y\in U$,
$$
\dist(x,\partial U)-\dist(y,\partial U)\leq \|x-y\|.
$$
Hence, the length of $h_m$ with respect to the usual Euclidean metric is bounded below by $2^{-m}-2^{-(m+1)}=2^{-(m+1)}$. Moreover, the length of $h_m$ with respect to $g_U$ is bounded below by this Euclidean length multiplied by
$$
\inf_{x\in V_{2^{-m}}\backslash V_{2^{-(m+1)}}}\frac{1}{\min\{1,[\dist(x,\partial U)]\}}
\geq \frac{1}{2^{-m}}=2^m.
$$
Hence, the length of $h$ with respect to $g_U$ is bounded below by
$
\sum_{m=N}^{n-1}1/2= ({n-N})/{2}.
$
Taking the infimum over all such paths $h$, we see that $d_U(x_0,x_n)\geq (n-N)/2$. Hence, $x_n\notin S_{r}(x_0)$ for sufficiently large $n$, a contradiction.
\end{proof}

We can now define the metric $d$ on $\mathcal{M}$, the set of relatively closed subsets of $U$. We first define it for non-empty sets.

\begin{definition}[Attouch--Wets metric on non-empty closed sets {\cite[page 79]{beer1993topologies}}]
\label{def:AW_top1}
Let $x_0\in U$ and let $\mathcal{M}_0$ be the set of non-empty closed subsets of $U$. The Attouch--Wets metric (with respect to the metric $d_U$) is
\begin{equation}\label{CHAP1_eq:attouch-wets-metric}
\dAW(X,Y)=\sum_{n=1}^{\infty} 2^{-n}\min\left\{1,\sup_{x\in S_n(x_0)}\left|d_U(x,X){-}d_U(x,Y)\right|\right\},\quad X,Y{\in}\mathcal{M}_0,
\end{equation}
where $d_U(x,X)=\inf_{x'\in X}d_U(x,x')$.
\end{definition}

The following is an intuitive characterization of this topology \cite[Theorem 3.1.7]{beer1993topologies}. Let
$$
e_{U}(X,Y)=\sup_{x\in X} d_U(x,Y),\quad X,Y{\in}\mathcal{M}_0.
$$
For any $\{X_n\}\subset \mathcal{M}_0$ and $X\in \mathcal{M}_0$, $\lim_{n\rightarrow\infty}\dAW(X_n,X)=0$ if and only if
$$
\lim_{n\rightarrow\infty} e_{U}(X\cap S_m(x_0),X_n)
= \lim_{n\rightarrow\infty} e_{U}(X_n\cap S_m(x_0),X) =0
\quad\forall m\in\mathbb{N}.
$$
Hence, convergence with respect to $\dAW$ amounts to locally uniform convergence of subsets of $U$. That is, avoiding spectral pollution and spectral invisibility.

\begin{remark}
In general, $\max\{e_{U}(X\cap S_m(x_0),Y),e_{U}(Y\cap S_m(x_0),X)\}$ need not obey the triangle inequality and, hence, we do not define $\dAW$ directly in terms of $e_U$.
\end{remark}

The above characterization of convergence in the Attouch--Wets sense also allows us to extend convergence to include the empty set. We say that $X_n\in \mathcal{M}$ converges to $\emptyset$ if, for each $m\in\mathbb{N}$, we have $X_n\cap S_m(x_0)=\emptyset$ for sufficiently large $n$.  By a slight abuse of terminology, we refer to the resulting metric space as the Attouch--Wets metric space $(\MAW,\dAW)$, accommodating the empty set.

\subsection{The evaluation set}\label{grid_def}

The final ingredient needed to define our computational problem is the evaluation set, i.e., the information that we allow our algorithms access to. Given a domain $U\subset \mathbb{C}$, we let $\{G_n\}_{n\in\mathbb{N}}$ be a sequence of sets that satisfy the following:
\begin{itemize}
\item Each set $G_n$ is a finite subset of $U$;
	\item Any $z\in G_n$ is complex rational, i.e., $\mathrm{Re}(z)\in\mathbb{Q}$ and $\mathrm{Im}(z)\in\mathbb{Q}$;
	\item $\lim_{n\rightarrow\infty}\dist(z,G_n)=0$ for any $z\in U$.
\end{itemize}
One may think of these sets as a grid of rational points that densely fills $U$ in the limit $n\rightarrow\infty$. For $T\in\Omega_{\mathrm{NL}}^U$, we then define the following two sets of evaluation functions:
\begin{align*}
\Lambda_1&=\{T\mapsto \langle T(z)e_j,\hat{e}_i \rangle_{\mathcal{H}_2}:z\in G_n\text{ and }i,j,n \in\mathbb{N} \}\\
\Lambda_2&=\Lambda_1\cup\{T\mapsto \langle T(z)e_j, T(z)e_i \rangle_{\mathcal{H}_2},
T\mapsto \langle T(z)^*\hat{e}_j,T(z)^*\hat{e}_i \rangle_{\mathcal{H}_1}
:z\in G_n\text{ and }i,j,n \in\mathbb{N} \}.
\end{align*}
The first evaluation set means that we have access to the matrix representation of $T(z)$ for any $z\in\cup_n G_n$. One can think of the second evaluation set as appending this information with access to the matrix elements of $T(z)^*T(z)$ and $T(z)T(z)^*$.\\

We can now precisely state our computational goal:
\begin{center}
\noindent\textbf{Goal:} Our goal is to solve $\{\spec,\Omega_{\mathrm{NL}}^U,\MAW,\Lambda_j\}$ and $\{\spec_\epsilon,\Omega_{\mathrm{NL}}^U,\MAW,\Lambda_j\}$ for $j=1,2$.
\end{center}

\section{Convergent universal algorithms}
\label{sec:convergent_algs}

We now construct our algorithms, establish their convergence properties, and prove that they yield optimal classifications within the SCI hierarchy.

\subsection{Nonlinear injection moduli}

We shall use injection moduli to compute spectra.

\begin{definition}
The injection modulus of $A\in\mathcal{C}(\mathcal{H}_1,\mathcal{H}_2)$ is defined as
$$
\sigma_{\mathrm{inf}}(A)=\inf\{\left\|Ax\right\|:x\in\dom(A),\left\|x\right\|=1\}=\sup\{\tau\geq 0:\left\|Ax\right\|\geq \tau \|x\|\,\forall x\in\dom(A)\}.
$$
\end{definition}

The following perturbation bound establishes the continuity of $\sigma_{\inf}$ on $\mathcal{C}(\mathcal{H}_1,\mathcal{H}_2)$ with respect to the topology induced by $\hat{\delta}$.

\begin{proposition}\label{prop_sigma_inf_cts}[Perturbation bound for nonlinear injection moduli]
Suppose that $A,B\in\mathcal{C}(\mathcal{H}_1,\mathcal{H}_2)$ with
$
\hat{\delta}(A,B)\sqrt{1+[\sigma_{\mathrm{\inf}}(A)]^2}<1.
$
Then
$$
\sigma_{\mathrm{\inf}}(B)\leq\sigma_{\mathrm{\inf}}(A)+\frac{2\hat{\delta}(A,B)\sqrt{1+[\sigma_{\mathrm{\inf}}(A)]^2}}{1-\hat{\delta}(A,B)\sqrt{1+[\sigma_{\mathrm{\inf}}(A)]^2}}.
$$
In particular, if $\lim_{n\rightarrow\infty}\hat{\delta}(A_n,A)=0$, then $\lim_{n\rightarrow\infty}\sigma_{\mathrm{\inf}}(A_n)=\sigma_{\mathrm{\inf}}(A)$.
\end{proposition}

\begin{proof}
Let $\epsilon>0$. Then there exists $x\in\dom(A)$ with $\|x\|=1$ and $\|Ax\|\leq \sigma_{\mathrm{\inf}}(A)+\epsilon$. Let $\phi = (x,Ax)/\sqrt{1+\|Ax\|^2}\in\mathrm{gr}(A)$, then there exists $\psi=(y,By)\in\mathrm{gr}(B)$ with $\|\phi-\psi\|\leq \hat{\delta}(A,B)+\epsilon$. Hence,
$$
\|By\|\leq \|Ax\|/\sqrt{1+\|Ax\|^2}+\hat{\delta}(A,B)+\epsilon,\quad \|y\|\geq \|x\|/\sqrt{1+\|Ax\|^2}-\hat{\delta}(A,B)-\epsilon.
$$
It follows that
$$
\sigma_{\mathrm{\inf}}(B)\leq \frac{\|Ax\|/\sqrt{1+\|Ax\|^2}+\hat{\delta}(A,B)+\epsilon}{1/\sqrt{1+\|Ax\|^2}-\hat{\delta}(A,B)-\epsilon}\leq \frac{\sigma_{\mathrm{\inf}}(A)/\sqrt{1+[\sigma_{\mathrm{\inf}}(A)]^2}+\hat{\delta}(A,B)+2\epsilon}{1/\sqrt{1+[\sigma_{\mathrm{\inf}}(A)+\epsilon]^2}-\hat{\delta}(A,B)-\epsilon}.
$$
Taking $\epsilon\downarrow 0$, we obtain
$$
\sigma_{\mathrm{\inf}}(B)\leq \frac{\sigma_{\mathrm{\inf}}(A)+\hat{\delta}(A,B)\sqrt{1+[\sigma_{\mathrm{\inf}}(A)]^2}}{1-\hat{\delta}(A,B)\sqrt{1+[\sigma_{\mathrm{\inf}}(A)]^2}}=\sigma_{\mathrm{\inf}}(A)+\frac{2\hat{\delta}(A,B)\sqrt{1+[\sigma_{\mathrm{\inf}}(A)]^2}}{1-\hat{\delta}(A,B)\sqrt{1+[\sigma_{\mathrm{\inf}}(A)]^2}}.
$$
Now suppose that $\lim_{n\rightarrow\infty}\hat{\delta}(A_n,A)=0$. Taking $B=A_n$, the above bound shows that
$
\limsup_{n\rightarrow\infty}\sigma_{\mathrm{\inf}}(A_n)\leq  \sigma_{\mathrm{\inf}}(A).
$
It follows that $\hat{\delta}(A_n,A)\sqrt{1+[\sigma_{\mathrm{\inf}}(A_n)]^2}<1$ for sufficiently large $n$, and, hence, we may reverse the roles of $A$ and $A_n$ to see that
$$
\sigma_{\mathrm{\inf}}(A)\leq\sigma_{\mathrm{\inf}}(A_n)+\frac{2\hat{\delta}(A_n,A)\sqrt{1+[\sigma_{\mathrm{\inf}}(A_n)]^2}}{1-\hat{\delta}(A_n,A)\sqrt{1+[\sigma_{\mathrm{\inf}}(A_n)]^2}}.
$$
It follows that $\liminf_{n\rightarrow\infty}\sigma_{\mathrm{\inf}}(A_n)\geq  \sigma_{\mathrm{\inf}}(A)$.
\end{proof}

\begin{remark}[Discontinuity of $\sigma_{\mathrm{\inf}}$ under weaker assumptions of continuity of $z\mapsto T(z)$]
The gap $\hat{\delta}$ between closed operators and generalized convergence extends the notion of convergence for bounded operators in the operator norm. A natural question is whether $\sigma_{\mathrm{\inf}}$ remains continuous under weaker assumptions on the continuity of $z\mapsto T(z)$. However, the following example shows that continuity of $z\mapsto T(z)$ with respect to the strong operator topology alone does not guarantee continuity of $\sigma_{\mathrm{\inf}}$.
\end{remark}

\begin{example}
\label{exam_weaker_false}
Let $\mathcal{H}_1=\mathcal{H}_2=L^2(\mathbb{R})$, $U=\mathbb{C}$ and $\phi$ be a smooth, non-negative, compactly supported bump function on $\mathbb{R}$ taking maximum value 1. Set $T(z)$ as multiplication by $1-\phi(\cdot -1/|z|)$ and $T(0)=I$. Let $\{z_n\}\subset \mathbb{C}$ with $\lim_{n\rightarrow\infty}z_n=z\in\mathbb{C}$. Then $T(z_n)$ converges strongly to $T(z)$ as $n\rightarrow\infty$ by the dominated convergence theorem. If we take $z=0$ and $z_n\neq 0$, $\sigma_{\mathrm{inf}}(T(z_n))=0$ but $\sigma_{\mathrm{inf}}(T(z))=1$.
\end{example}

The following simple result relates injection moduli to the norms of inverses. Although this connection is well-known as a ``folklore'' result, we include a proof for completeness.

\begin{lemma}[Injection moduli and norms of inverses]\label{lemma:injection1}
Let $A\in\mathcal{C}(\mathcal{H}_1,\mathcal{H}_2)$ be densely defined. If $A$ is boundedly invertible, then
$
\sigma_{\mathrm{inf}}(A)=\sigma_{\mathrm{inf}}(A^*)=\|A^{-1}\|^{-1}.
$
If $A$ is not boundedly invertible, then
$
\min\{\sigma_{\mathrm{inf}}(A),\sigma_{\mathrm{inf}}(A^*)\}=0.
$
\end{lemma}

\begin{proof}
Suppose first that $A$ is boundedly invertible. We prove that $\sigma_{\mathrm{inf}}(A)=\|A^{-1}\|^{-1}$ and the other case is similar using the fact that $\|A^{-1}\|=\|(A^*)^{-1}\|$. Let $x\in\dom(A)$ with $\|x\|=1$, then
$$
1=\|x\|=\left\|A^{-1}Ax\right\|\leq\left\|A^{-1}\right\|\left\|Ax\right\|.
$$
Upon taking the infimum over $x$, $\sigma_{\mathrm{inf}}(A)\geq \|A^{-1}\|^{-1}$. Conversely, let $x_n\in \mathcal{H}_2$ such that $\|x_n\|=1$ and $\|A^{-1}x_n\|\rightarrow \|A^{-1}\|$. Then
$$
1=\|x_n\|=\left\|AA^{-1}x_n\right\|\geq\sigma_{\mathrm{inf}}(A)\left\|A^{-1}x_n\right\|.
$$
Letting $n\rightarrow\infty$, we obtain $\sigma_{\mathrm{inf}}(A)\leq \|A^{-1}\|^{-1}$.

Suppose that $A$ is not boundedly invertible but, for a contradiction,
$
\min\{\sigma_{\mathrm{inf}}(A),\sigma_{\mathrm{inf}}(A^*)\}>0.
$
It follows that $\ker(A^*)=\{0\}$. Consequently, the range of $A$, $\ran(A)$, is dense in $\mathcal{H}_2$ since $\ran(A)^{\perp}=\ker(A^*)$. Furthermore, $A$ is injective on its domain. Therefore, we can define $A^{-1}$ on the dense subspace $\ran(A)$. This operator is bounded since $\sigma_{\mathrm{inf}}(A) > 0$, allowing us to extend it to a bounded operator on the whole of $\mathcal{H}_2$. The closedness of $A$ implies that $AA^{-1}=I$. Clearly, $A^{-1}Ax=x$ for all $x\in\dom(A)$. Hence, $A$ is boundedly invertible, a contradiction.
\end{proof}

For $T\in\Omega_{\mathrm{NL}}^U$, we therefore define the function
$$
\gamma(z,T)=\min\left\{\sigma_{\mathrm{inf}}(T(z)),\sigma_{\mathrm{inf}}(T(z)^*)\right\}.
$$
An immediate consequence of \cref{lemma:injection1} is that
\begin{equation}
\label{CHAP1_eq:spectra_char}
\spec(T)=\{z\in U:\gamma(z,T)=0\},\quad\spec_{\epsilon}(T)=\mathrm{Cl}_U(\{z\in U:\gamma(z,T)<\epsilon\}).
\end{equation}
Moreover, \cref{prop_sigma_inf_cts} and the fact that $\hat{\delta}(S^*,T^*)=\hat{\delta}(S,T)$ for densely defined $S,T$ show that the map $z\mapsto\gamma(z,T)=\|T(z)^{-1}\|^{-1}$ is continuous.

\subsection{Computing nonlinear injection moduli}

To computationally exploit \cref{CHAP1_eq:spectra_char}, we must approximate the function $\gamma$. For any $n\in\mathbb{N}$, we let $\mathcal{P}_{n}$ denote the orthogonal projection onto $\mathrm{span}\{e_1,\ldots,e_n\}$ (with domain $\mathcal{H}_1$) and $\mathcal{Q}_{n}$ denote the orthogonal projection onto $\mathrm{span}\{\hat{e}_1,\ldots,\hat{e}_n\}$ (with domain $\mathcal{H}_2$). For $n_1,n_2\in\mathbb{N}$, let
\begin{align}
\gamma_{n_2,n_1}(z,T)&=\min\left\{\sigma_{\mathrm{inf}}(\mathcal{Q}_{n_1}T(z)\mathcal{P}_{n_2}^*),\sigma_{\mathrm{inf}}(\mathcal{P}_{n_1}T(z)^*\mathcal{Q}_{n_2}^*)\right\},\label{CHAP1_eq:gamman1n2a}\\
\gamma_{n_2}(z,T)&=\min\left\{\sigma_{\mathrm{inf}}(T(z)\mathcal{P}_{n_2}^*),\sigma_{\mathrm{inf}}(T(z)^*\mathcal{Q}_{n_2}^*)\right\}.\label{CHAP1_eq:gamman1n2b}
\end{align}
We will use the following monotonicity properties of $\gamma_{n_2,n_1}$ and $\gamma_{n_2}$. The proof of the following lemma is almost identical to the linear case, but, again, we include it for completeness.

\begin{lemma}[Monotonicity of injection moduli]
\label{CHAP1_lem:monotonicity_gamma}\index{injection modulus!monotonicity}
For any $T\in\Omega_{\mathrm{NL}}^U$ and $z\in U$, the following hold:
\begin{itemize}[leftmargin=0.7cm]
\item[\rm(i)] $\lim_{n_1\rightarrow\infty}\gamma_{n_2,n_1}(z,T)=\gamma_{n_2}(z,T)$, where the convergence is monotonic from below;
	\item[\rm(ii)] $\lim_{n_2\rightarrow\infty}\gamma_{n_2}(z,T)=\gamma(z,T)$, where the convergence is monotonic from above.
\end{itemize}
\end{lemma}

\begin{proof}
To prove (i), consider the functions $z\mapsto \sigma_{\mathrm{inf}}(\mathcal{Q}_{n_1}T(z)\mathcal{P}_{n_2}^*)$. As $n_1$ increases, the range of $\mathcal{Q}_{n_1}T(z)\mathcal{P}_{n_2}^*$ increases and, hence, its injection modulus cannot decrease. We also have $\sigma_{\mathrm{inf}}(\mathcal{Q}_{n_1}T(z)\mathcal{P}_{n_2}^*)\leq \sigma_{\mathrm{inf}}(T(z)\mathcal{P}_{n_2}^*)$. Fix $z\in\mathbb{C}$. Then
$$
\lim_{n_1\rightarrow\infty}\|\mathcal{Q}_{n_1}^*\mathcal{Q}_{n_1}T(z)\mathcal{P}_{n_2}^*-T(z)\mathcal{P}_{n_2}^*\|=0.
$$
Since $\sigma_{\mathrm{inf}}$ is continuous with respect to the operator norm topology when restricted to bounded operators, the convergence follows. Arguing similarly for $\sigma_{\mathrm{inf}}(\mathcal{P}_{n_1}T(z)^*\mathcal{Q}_{n_2}^*)$, part (i) follows.

For part (ii), we use the fact that the domains of $T(z)\mathcal{P}_{n_2}^*$ and $T(z)^*\mathcal{Q}_{n_2}^*$ are increasing in $n_2$ and, hence, $\gamma_{n_2}(z,T)$ is decreasing. Fix $z\in\mathbb{C}$ and $\epsilon>0$. We may choose $x\in\mathcal{D}(T(z))$ of unit norm with $\|T(z)x\|\leq \sigma_{\mathrm{inf}}(T(z))+\epsilon$. Since $\mathrm{span}\{e_n:n\in\mathbb{N}\}$ forms a core of $T(z)$, we may choose $x'\in\mathrm{span}\{e_n:n\in\mathbb{N}\}$ of unit norm such that $\|T(z)(x-x')\|\leq \epsilon$. Since $x'\in\mathrm{span}\{e_1,\ldots,e_{n_2}\}$ for sufficiently large $n_2$, it follows that
$$
\smash{\limsup_{n_2\rightarrow\infty} \sigma_{\mathrm{inf}}(T(z)\mathcal{P}_{n_2}^*)}\leq \|T(z)x'\|\leq\|T(z)x\|+\|T(z)(x-x')\|\leq\sigma_{\mathrm{inf}}(T(z))+2\epsilon.
$$
Since $\epsilon>0$ was arbitrary, $\lim_{n_2\rightarrow\infty} \sigma_{\mathrm{inf}}(T(z)\mathcal{P}_{n_2}^*)=\sigma_{\mathrm{inf}}(T(z))$. We can argue similarly for $\sigma_{\mathrm{inf}}(T(z)^*\mathcal{Q}_{n_2}^*)$, and, hence, part (ii) follows.
\end{proof}

We now have the following proposition, which provides a route to computing $\gamma$.

\begin{proposition}\label{approx_gamma}
Let $U\subset\mathbb{C}$ be a domain and $z\in \cup_{n}G_n$, where $\{G_n\}$ is a set of grids satisfying the properties outlined in \cref{grid_def}. For any $\epsilon>0$, the following hold:
\begin{itemize}
	\item For any input $T\in\Omega_{\mathrm{NL}}^U$, we can compute an $\epsilon$-accurate approximation of $\gamma_{n_2,n_1}(z,T)$ using $\Lambda_1$ and finitely many arithmetic operations and comparisons;
	\item For any input $T\in\Omega_{\mathrm{NL}}^U$, we can compute an $\epsilon$-accurate approximation of $\gamma_{n_2}(z,T)$ using $\Lambda_2$ and finitely many arithmetic operations and comparisons
\end{itemize}
\end{proposition} 

\begin{proof}
For the first statement, we view $\mathcal{Q}_{n_1}T(z)\mathcal{P}_{n_2}^*$ as a $n_1\times n_2$ matrix through the inner products $\langle T(z)e_j,\hat{e}_i \rangle_{\mathcal{H}_2}$. In a similar fashion, we view $\mathcal{P}_{n_1}T(z)^*\mathcal{Q}_{n_2}^*$ as a $n_1\times n_2$ matrix. Assuming access to these inner products though $\Lambda_1$, computing $\gamma_{n_2,n_1}(z,T)$ consists of computing the smallest singular values of the two \textit{finite rectangular} matrices, which can be done using arithmetic operations to any desired accuracy \cite[Theorem 6.7]{colbrook3}.

For the second statement, we use the fact that
$$
\sigma_{\mathrm{inf}}(T(z)\mathcal{P}_{n_2}^*)=\sqrt{\sigma_{\mathrm{inf}}(\mathcal{P}_{n_2}T(z)^*T(z)\mathcal{P}_{n_2}^*)}.
$$
We can obtain arbitrarily accurate approximations of $\sigma_{\mathrm{inf}}(\mathcal{P}_{n_2}T(z)^*T(z)\mathcal{P}_{n_2}^*)$ using the inner products in $\Lambda_2$. We then form approximations of the square roots. A similar argument applies to $\sigma_{\mathrm{inf}}(T(z)^*\mathcal{Q}_{n_2}^*)$ and, hence, to $\gamma_{n_2}(z,T)$.
\end{proof}

\begin{remark}[Rectangular truncations replacing $\Lambda_2$]
In applications, one can often select $n_1=f(n_2)$, for instance when the matrix representation of $T(z)$ is sparse (i.e., has finitely many nonzero entries per row and column). However, one can also show that, in general, this choice is not possible without additional information, such as that provided by $\Lambda_2$.
\end{remark}

\subsection{Algorithms}
\label{sec:basic_convergence}

We can now construct our algorithms and establish their convergence properties.

\begin{remark}[Computation of singular values]
In what follows, we write the algorithms for exact evaluation of each $\gamma_{n_2,n_1}(z,T)$ in the case of $\Lambda_1$ and $\gamma_{n}(z,T)$ in the case of $\Lambda_2$. However, by applying \cref{approx_gamma}, we can similarly derive arithmetic algorithms that approximate $\gamma_{n_2,n_1}(z,T)$ \textit{from below} to an accuracy $1/n_1$ (for instance, by computing an approximation accurate to $1/(2n_1)$ and then subtracting $1/(2n_1)$), and analogously approximate $\gamma_{n}(z,T)$ \textit{from above}. This modification does not affect the subsequent arguments.
\end{remark}

\begin{algorithm}[t]
\textbf{Input:} Operator pencil $T\in\Omega_{\mathrm{NL}}^U$, $\epsilon>0$, grids $\{G_n\}$, $n_1,n_2\in\mathbb{N}$.
\begin{algorithmic}[1]
\State For each $z\in G_{n_2}$, compute $\gamma_{n_2,n_1}(z,T)$ using $\Lambda_1$.
\State Set $\Gamma_{n_2,n_1}^\epsilon(T)=\left\{z\in G_{n_2}:\gamma_{n_2,n_1}(z,T)+\frac{1}{n_2}\leq \epsilon\right\}$.
\end{algorithmic} \textbf{Output:} $\Gamma_{n_2,n_1}^\epsilon(T)$, where $\lim\limits_{n_2\rightarrow\infty}\lim\limits_{n_1\rightarrow\infty}\Gamma_{n_2,n_1}^\epsilon(T)=\spec_\epsilon(T)$.
\caption{Two-limit convergent algorithm for computing pseudospectra.}
\label{alg1}
\end{algorithm}

Our first algorithm, \cref{alg1}, uses $\Lambda_1$ to compute pseudospectra. Since each grid $G_n$ is finite and $\gamma_{n_2,n_1}$ is non-decreasing in $n_1$, we have
$$
\lim_{n_1\rightarrow\infty}\Gamma_{n_2,n_1}^\epsilon(T)=\left\{z\in G_{n_2}:\gamma_{n_2}(z,T)+\frac{1}{n_2}\leq \epsilon\right\}=:\Gamma_{n_2}^\epsilon(T).
$$
Since $\gamma_{n_2}(z,T)\geq \gamma(z,T)$, it follows that if $z\in\Gamma_{n_2}^\epsilon(T)$, then $\gamma(z,T)\leq\gamma_{n_2}(z,T)<\epsilon$, and thus
$$
\Gamma_{n_2}^\epsilon(T)\subset \spec_{\epsilon}(T).
$$
We further show that $\lim_{n_2\rightarrow\infty}\Gamma_{n_2}^\epsilon(T)=\spec_{\epsilon}(T)$, with convergence in the Attouch--Wets topology constructed in \cref{sec:metric_def}.

\begin{proposition}[Convergence to pseudospectra]\label{prop_pseudospec}
Let $U\subset\mathbb{C}$ be a domain and $T\in\Omega_{\mathrm{NL}}^U$, then
$$
\lim_{n_2\rightarrow\infty}\Gamma_{n_2}^\epsilon(T)=\spec_\epsilon(T),
$$
with convergence in the Attouch--Wets topology.
\end{proposition}

\begin{proof}
Suppose for a contradiction that this were false. The inclusion $\Gamma_{n_2}^\epsilon(T)\subset \spec_{\epsilon}(T)$ means that we may assume without loss of generality that $\spec_{\epsilon}(T)\neq\emptyset$ (otherwise convergence clearly holds). By the characterization of the Attouch--Wets topology, there exists an $m\in\mathbb{N}$ such that
$$
\limsup_{n_2\rightarrow\infty} e_{U}(\spec_{\epsilon}(T)\cap S_m(x_0),\Gamma_{n_2}^\epsilon(T))>0.
$$
Without loss of generality, we may assume that there exists $z_{n_2}\in\spec_{\epsilon}(T)\cap S_m(x_0)$ and $\delta>0$ such that $d_U(z_{n_2},\Gamma_{n_2}^\epsilon(T))\geq \delta$. We now use \cref{balls_separated} to see that the closure of $\spec_{\epsilon}(T)\cap S_m(x_0)$ is contained in the set $U$ and separated from its boundary $\partial U$. This allows us to assume two things without loss of generality. First, $\dist(z_{n_2},\Gamma_{n_2}^\epsilon(T))\geq \delta$ (by taking $\delta$ smaller if necessary). Second, $\lim_{n_2\rightarrow\infty}z_{n_2}=z\in U$ (by taking a subsequence if necessary). 

By definition of $\spec_{\epsilon}(T)$, we may choose $w\in U$ such that $|z-w|\leq \delta/4$ and $\gamma(w,T)<\epsilon$. Since $\gamma(\cdot,T)$ is continuous, there exists $\eta<\delta/2$ such if $w'$ has $|w'-w|<\eta$, then $w'\in U$ and $\gamma(w',T)+\eta<\epsilon$. For $n_2\geq N$ for some $N$, there exists $w_{n_2}=w_N\in G_{n_2}$ with $|w_{n_2}-w|<\eta$. Hence,
$$
\lim_{n_2\rightarrow\infty}\gamma_{n_2}(w_N,T)+\frac{1}{n_2}=\gamma(w_N,T)<\epsilon.
$$
It follows that $w_{N}\in\Gamma_{n_2}^\epsilon(T)$ for sufficiently large $n_2$.
But then $|w_{N}-z|\leq |w_{N}-w|+|w-z|\leq 3\delta/4$. Hence, $\limsup_{n_2\rightarrow\infty}|w_{N}-z_{n_2}|\leq 3\delta/4$, which contradicts $\dist(z_{n_2},\Gamma_{n_2}^\epsilon(T))\geq \delta$.
\end{proof}

\begin{algorithm}[t]
\textbf{Input:} Operator pencil $T\in\Omega_{\mathrm{NL}}^U$, grids $\{G_n\}$, $n_1,n_2,n_3\in\mathbb{N}$.
\begin{algorithmic}[1]
\State Set $\Gamma_{n_3,n_2,n_1}(T)=\Gamma_{n_2,n_1}^{1/n_3}(T)$, where $\Gamma_{n_2,n_1}^\epsilon$ is the output of \cref{alg1}.
\end{algorithmic} \textbf{Output:} $\Gamma_{n_3,n_2,n_1}(T)$, where $\lim\limits_{n_3\rightarrow\infty}\lim\limits_{n_2\rightarrow\infty}\lim\limits_{n_1\rightarrow\infty}\Gamma_{n_3,n_2,n_1}(T)=\spec(T)$.
\caption{Three-limit convergent algorithm for computing spectra.}
\label{alg2}
\end{algorithm}

It follows that
\begin{equation}\label{convergence2}
\lim_{n_2\rightarrow\infty}\lim_{n_1\rightarrow\infty}\Gamma_{n_2,n_1}^\epsilon(T)=\spec_\epsilon(T).
\end{equation}
To compute spectra, we compute pseudospectra for a sequence of decreasing $\epsilon$, as outlined in \cref{alg2}. The following proposition shows that $\spec_\epsilon(T)$ converges to $\spec(T)$ as $\epsilon\downarrow0$.

\begin{proposition}[Convergence to spectra]
Let $U\subset\mathbb{C}$ be a domain and $T\in\Omega_{\mathrm{NL}}^U$, then
$$
\lim_{\epsilon\downarrow 0}\spec_\epsilon(T)=\spec(T),
$$
with convergence in the Attouch--Wets topology.
\end{proposition}

\begin{proof}
We have $\spec(T)\subset\spec_\epsilon(T)$ for any $\epsilon>0$ and clearly $\spec_{\epsilon'}(T)\subset \spec_\epsilon(T)$ for $\epsilon'<\epsilon$. Suppose for a contradiction that the convergence in the proposition does not hold. Then there exists some $r>0$, $\delta>0$, $\epsilon_n\downarrow 0$, and $z_n$ such that
$$
z_n\in \spec_{\epsilon_n}(T)\cap S_r(x_0)\quad\text{and}\quad \dist(z_n,\spec(T))\geq \delta.
$$
Arguing as in the prof if \cref{prop_pseudospec} (using \cref{balls_separated}), we may assume without loss of generality that $\lim_{n\rightarrow\infty}z_n=z\in U$. Since $\gamma(\cdot,T)$ is continuous, we must have $\gamma(z,T)=0$ and, hence, $z\in\spec(T)$, a contradiction.
\end{proof}

It follows that
\begin{equation}\label{convergence1}
\lim_{n_3\rightarrow\infty}\lim_{n_2\rightarrow\infty}\lim_{n_1\rightarrow\infty}\Gamma_{n_3,n_2,n_1}(T)=\spec(T).
\end{equation}
In other words, the output of \cref{alg2} converges to the spectrum.

\subsection{Classifications of problems and algorithmic optimality}

We now classify our computational problems and prove that \cref{alg1,alg2} are optimal.

\begin{theorem}[SCI classification of nonlinear spectral problems]\label{SCI_bounds1}
Let $U\subset\mathbb{C}$ be a domain. Then
$$
\Delta_3^G\not\ni\left\{\spec,\Omega_{\mathrm{NL}}^U,\MAW,\Lambda_1\right\}\in \Pi_3^A,\quad \Delta_2^G\not\ni\left\{\spec,\Omega_{\mathrm{NL}}^U,\MAW,\Lambda_2\right\}\in \Pi_2^A
$$
and for any $\epsilon>0$,
$$
\Delta_2^G\not\ni\left\{\spec_\epsilon,\Omega_{\mathrm{NL}}^U,\MAW,\Lambda_1\right\}\in \Sigma_2^A,\quad \Delta_1^G\not\ni\left\{\spec_\epsilon,\Omega_{\mathrm{NL}}^U,\MAW,\Lambda_2\right\}\in \Sigma_1^A.
$$
\end{theorem}

\begin{proof}
The upper bounds for $\Lambda_1$ follow from \cref{convergence1,convergence2} and the inclusions discussed in the previous subsection. For the upper bounds for $\Lambda_2$, we use \cref{approx_gamma} to bypass the first limit when approximating $\gamma(z,T)$ in \cref{alg1,alg2}. For the lower bounds, we argue for $\mathcal{H}_1=\mathcal{H}_2=l^2(\mathbb{N})$, and the general case is entirely analogous after identifying operators with infinite matrices with respect to the bases of $\mathcal{H}_1$ and $\mathcal{H}_2$. Given a domain $U$, we may choose a closed (Euclidean) ball $V=B_r(z_0)\subset U$. Let $A$ be a bounded linear operator acting on $l^2(\mathbb{N})$ with $\|A\|\leq 1$. We set
$$
T(z)= A-\frac{1}{r}(z-z_0)I.
$$
Using the fact that $\|A\|\leq 1$, we have $\spec(T)=r\cdot\spec(A)+z_0$. Hence, the existence of towers of algorithms for $\spec(T)$ implies the existence of corresponding towers (with the same SCI upper bounds from the towers) for $\spec(A)$. The proof is finished by lifting the lower bounds for computing spectra and pseudospectra of such $A$ (e.g., \cite{colbrook2020PhD,ben2015can}).
\end{proof}

We end this section with lower bounds for Hermitian operator pencils.

\begin{definition}[Hermitian pencils]
Let $U\subset \mathbb{C}$ be a domain that is symmetric about the real axis (i.e., if $z\in U$, then $\overline{z}\in U$), $\mathcal{H}_1=\mathcal{H}_2$, and $T\in\Omega_{\mathrm{NL}}^U$. We say that $T$ is Hermitian if $T(\overline{z})=T(z)^*$ for all $z\in U$. We denote the class of such pencils by $\Omega_{\mathrm{NL}}^{U,H}$
\end{definition}

If $T$ is Hermitian, then both $\spec(T)$ and $\spec_\epsilon(T)$ are symmetric about the real axis.
In the case of linear families described in \cref{linear_example}, the definition of a Hermitian pencil reduces to the standard definition for operators. In this setting, computing the spectrum has the same classification as computing the pseudospectrum (i.e., $\Sigma_2^A$ when using $\Lambda_1$ and $\Sigma_1^A$ when using $\Lambda_2$). Consequently, the task becomes easier than in the general case of bounded operators. It is somewhat surprising that this is no longer true in the nonlinear case, as demonstrated by the following theorem.

\begin{theorem}[The SCI classifications do not change for Hermitian problems]\label{hermitian_still hard}
Let $U\subset\mathbb{C}$ be a domain that is symmetric about the real axis. Then
$$
\left\{\spec,\Omega_{\mathrm{NL}}^{U,H},\MAW,\Lambda_1\right\}\not\in \Delta_3^G,\quad \left\{\spec,\Omega_{\mathrm{NL}}^{U,H},\MAW,\Lambda_2\right\}\not\in \Delta_2^G.
$$
\end{theorem}

To prove this proposition, we embed a decision problem. Let $\OS$ denote the class of (linear) self-adjoint operators on $l^2(\mathbb{N})$ for which $\mathrm{span}\{e_n:n\in\mathbb{N}\}$ forms a core. Consider the following decision problem:
$$
\Xi_0:A \rightarrow\text{``Is $0\in\spec(A)$?''}\quad A\in\OS.
$$
The following proposition classifies the difficulty of this problem. To facilitate the proof of \cref{hermitian_still hard}, we treat the problem function $\Xi_0$ mapping to the metric space $[0,1]$ instead of the discrete space $\{0,1\}$ typically used for decision problems.

\begin{proposition}\label{prop_embedding}
$\{\Xi_0,\OS,[0,1],\Lambda_1\}\not\in\Delta_3^G$.
\end{proposition}

\begin{proof}
Let $\Omega'$ denote the collection of all infinite matrices $a=\{a_{i,j}\}_{i,j\in\mathbb{N}}$ with entries $a_{i,j}\in\{0,1\}$. Define the problem function
$$
\Xi_{2,Q}(\{a_{i,j}\})=\begin{cases}
1,\quad&\text{if }\sum_{i}a_{i,j}=\infty\text{ for all but finitely many $j$},\\
0,\quad&\text{otherwise}.
\end{cases}
$$
In \cite{colbrook4}, it was shown through connections with descriptive set theory that $\{\Xi_{2,Q},\Omega',[0,1],\Lambda'\}\notin\Delta_3^G$, where $\Lambda'$ is the set of component-wise evaluations of $\{a_{i,j}\}$.

Suppose for a contradiction that $\{\Gamma_{n_2,n_1}\}$ is a $\Delta_3^G$-tower of algorithms for $\{\Xi_0,\OS,[0,1],\Lambda_1\}$. Given a sequence $\{c_i\}_{i\in\mathbb{N}}$ with each $c_i\in\{0,1\}$, list the indices where $c_i=1$ as
$$
\{i:c_i=1\}=\{i_1,i_2,\ldots\},\quad i_1\leq i_2\leq\cdots.
$$
This set may be empty, finite, or countably infinite. Define $i_0=0$ and let $r(\{c_i\})$ be one greater than the cardinality of $\{i:c_i=1\}$. If $r(\{c_i\})<\infty$, then we define
$$
l_k=2+i_k-i_{k-1},\quad k=1,\ldots,r(\{c_i\})-1,\quad l_{r(\{c_i\})}=\infty,
$$
otherwise we define
$$
l_k=2+i_k-i_{k-1},\quad k\in\mathbb{N}.
$$
We set\index{hidden corner entries!essential spectrum}
$$
C(\{c_i\})=\bigoplus_{k=1}^{r(\{c_i\})} A_{l_k},\quad A_{m}=\begin{pmatrix}
1& & & &1\\
 &0& & & \\
& &\ddots& & \\
 & & &0& \\
1& & & &1\\
\end{pmatrix}
\in\mathbb{C}^{m\times m},
$$
where $A_{\infty}=\diag(1,0,0,\ldots)$. Note that $\spec(C(\{c_i\}))\subset \{0,1,2\}$ and $1\in \spec(C(\{c_i\}))$ if and only if $r(\{c_i\})<\infty$. Given $\{a_{i,j}\}\in\Omega'$, define the self-adjoint operator
$$
A=A(\{a_{i,j}\})=\bigoplus_{j=1}^{\infty}\left[C(\{a_{i,j}\}_{i\in\mathbb{N}})- (1-2^{-(1+j)})I\right]\in\OS.
$$
If $\Xi_{2,Q}(\{a_{i,j}\})=1$, then $0\notin\spec(A)$.
If $\Xi_{2,Q}(\{a_{i,j}\})=0$, there exists a null sequence in $\spec(A)$ and, hence, $0\in\spec(A)$.
We define the following general algorithm that uses $\Lambda'$:
$$
\widehat\Gamma_{n_2,n_1}(\{a_{i,j}\})=1-\Gamma_{n_2,n_1}(A(\{a_{i,j}\})).
$$
The convergence of $\Gamma_{n_2,n_1}$ implies that
$$
\lim_{n_2\rightarrow\infty}\lim_{n_1\rightarrow\infty}\widehat\Gamma_{n_2,n_1}(\{a_{i,j}\})=\Xi_{2,Q}(\{a_{i,j}\}).
$$
However, this contradicts $\{\Xi_{2,Q},\Omega',[0,1],\Lambda'\}\notin\Delta_3^G$.
\end{proof}

\begin{proof}[Proof of \cref{hermitian_still hard}]
Suppose for a contradiction that $\{\spec,\Omega_{\mathrm{NL}}^{U,H},\MAW,\Lambda_2\}\in \Delta_2^G$ with a $\Delta_2^G$-tower of algorithms $\{\Gamma_n\}$. Consider the class $\OD$ of bounded diagonal self-adjoint operators on $l^2(\mathbb{N})$. For any $A\in \OD$, we may define a constant pencil $T_A\in\Omega_{\mathrm{NL}}^{U,H}$ by setting
$
\langle T_A(z)e_j,\hat{e}_i \rangle_{\mathcal{H}_2}=A_{ij},
$
where $A_{ij}$ is the $(i,j)$th matrix entry of $A$ with respect to the canonical basis of $l^2(\mathbb{N})$. Moreover, for this class of diagonal operators, $\Lambda_2$ is equivalent to the evaluation of each matrix entry $A_{ij}$. With this setup, it is easy to see that the decision problem of deciding whether $A$ is invertible does not lie in $\Delta_2^G$. We have $\spec(T_A)=\emptyset$ if $A$ is invertible and $\spec(T_A)=U$ otherwise. We then set
$
\tilde{\Gamma}_n(A)=1
$
if $\Gamma_n(T_A)\cap S_1(x_0)\neq\emptyset$ and $\tilde{\Gamma}_n(A)=0$ otherwise. Then $\{\tilde{\Gamma}_n\}$ is a $\Delta_2^G$-tower for deciding if elements of $\OD$ are invertible, a contradiction.

The proof that $\{\spec,\Omega_{\mathrm{NL}}^{U,H},\MAW,\Lambda_1\}\not\in \Delta_3^G$ is similar, except we now use \cref{prop_embedding}. Given $A\in \OS$, we can identity it with a constant pencil $T_A\in\Omega_{\mathrm{NL}}^{U,H}$ exactly as before. Suppose for a contradiction that $\{\Gamma_{n_2,n_1}\}$ is a $\Delta_3^G$-tower for $\{\spec,\Omega_{\mathrm{NL}}^{U,H},\MAW,\Lambda_1\}$. We now set
$$
a_{n_2,n_1}(A)=\inf_{z\in\Gamma_{n_2,n_1}(A)}\min\{d_U(z,x_0),1\},
$$
where the infimum over the empty set is taken to be $1$, and define
$$
\tilde{\Gamma}_{n_2,n_1}(A)=1-a_{n_2,n_1}(A).
$$
If $\Xi_0(A)=1$, then $\lim_{n_2\rightarrow\infty}\lim_{n_1\rightarrow\infty}a_{n_2,n_1}(A)=0$. On the other hand, if $\Xi_0(A)=0$, then $\lim_{n_2\rightarrow\infty}\lim_{n_1\rightarrow\infty}a_{n_2,n_1}(A)=1$. Hence, $\{\tilde{\Gamma}_{n_2,n_1}\}$ is a $\Delta_3^G$-tower for $\{\Xi_0,\OS,[0,1],\Lambda_1\}$, which contradicts \cref{prop_embedding}.
\end{proof}

\begin{remark}
The proof of \cref{hermitian_still hard} can easily be adapted in various ways—for instance, to using non-constant holomorphic families of type (A).
\end{remark}

\section{Computational examples}
\label{sec:examples}

We now provide several examples demonstrating the convergence of our methods, generalizations of the original setup, and the non-convergence of standard truncation methods, such as the finite section method. The code for all examples can be found at: \url{https://github.com/MColbrook/NL_spectra}.

\subsection{Nonlinear shifts}

Let $S$ be the bilateral shift $Se_j=e_{j+1}$ acting on $l^2(\mathbb{Z})$, where $\{e_n\}_{n\in\mathbb{Z}}$ are the canonical basis vectors. Let $U=\mathbb{C}$, $f:\mathbb{C}\rightarrow\mathbb{C}$ be continuous and set
$$
T(z)=S-f(z)S^*.
$$
We may write
$
T(z)=(S^2-f(z)I)S^*
$
and, hence,
$$
\spec(T)=\{z:f(z)\in\spec(S^2)\}=\{z:|f(z)|=1\}.
$$
However, if we form finite sections of $T$, we obtain a completely different (and wrong) spectrum. Let $\mathcal{P}^{\mathbb{Z}}_n$ denote the orthogonal projection onto $\mathrm{span}\{e_{-n},\ldots,e_n\}$ and set
$$
T_{2n+1}(z)=\mathcal{P}^{\mathbb{Z}}_n(S-f(z)S^*){\mathcal{P}^{\mathbb{Z}}_n}^*\in\mathbb{C}^{2n+1\times 2n+1}.
$$
Let $u_{2n+1}=((-1)^n[f(z)]^n,0,(-1)^{n-1}[f(z)]^{n-1},0,\ldots,-f(z),0,1)^\top$, then $T_{2n+1}(z)u_{2n+1}=0$ and, hence, $\spec(T_{2n+1})=\mathbb{C}$. Similarly, we could take an truncation of even dimension, in which case the corresponding Toeplitz matrix $T_{2n}(z)$ is invertible if and only if $f(z)\neq 0$ so that $\spec(T_{2n})=\{z\in\mathbb{C}:f(z)=0\}$. For example, \cref{fig:nonlinear_shift0} shows the pseudospectra of a $40\times 40$ truncation when $f(z)=\sin(4z)(|z|^2+1)$.

In contrast, we use rectangular truncations to form the approximations
$$
\gamma_{n}(z,T)=\min\left\{\sigma_{\mathrm{inf}}(\mathcal{P}^{\mathbb{Z}}_{n+1}T(z){\mathcal{P}^{\mathbb{Z}}_n}^*),\sigma_{\mathrm{inf}}(\mathcal{P}^{\mathbb{Z}}_{n+1}T(z)^*{\mathcal{P}^{\mathbb{Z}}_n}^*)\right\}.
$$
This leads to convergent $\Sigma_1^A$ and $\Pi_2^A$ for pseudospectra and the spectrum, respectively. For example, \cref{fig:nonlinear_shift} shows the approximation of pseudospectra when $f(z)=\sin(4z)(|z|^2+1)$ using $n=100$. (We can just as well take rectangular matrices with an even number of columns.) The spectrum is an infinite sequence of rings centred on the real axis.

\begin{figure}
\centering
\includegraphics[width=1\linewidth]{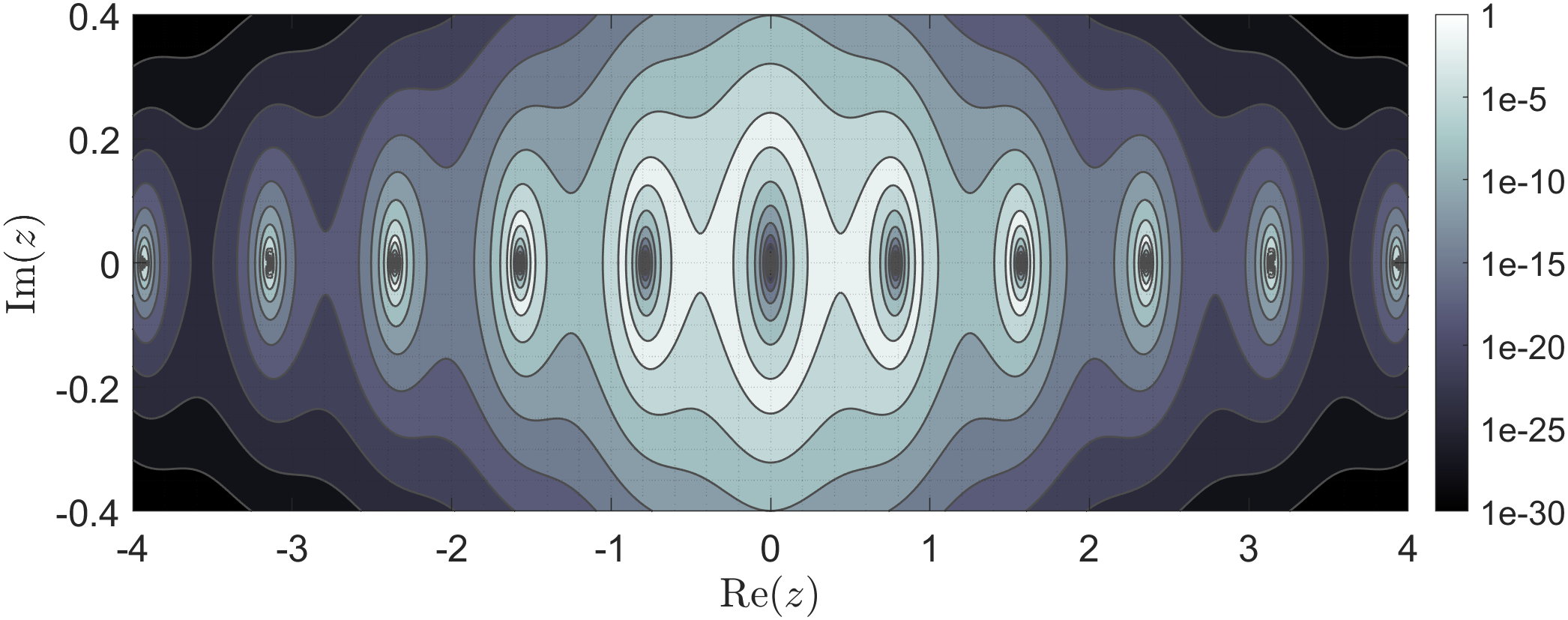}
\caption{Pseudosepctra of $40\times 40$ truncation of the nonlinear shift $S-f(z)S^*$ for $f(z)=\sin(4z)(|z|^2+1)$. The color corresponds to a logarithmic grid of $\epsilon$ values. The spectrum of the truncated problem is $\{\pi m/4:m\in\mathbb{Z}\}$, which is completely different to the spectrum $\spec(S-f(z)S^*)=\{z\in\mathbb{C}:|f(z)|=1\}$.\label{fig:nonlinear_shift0}}\vspace{2mm}
\includegraphics[width=1\linewidth]{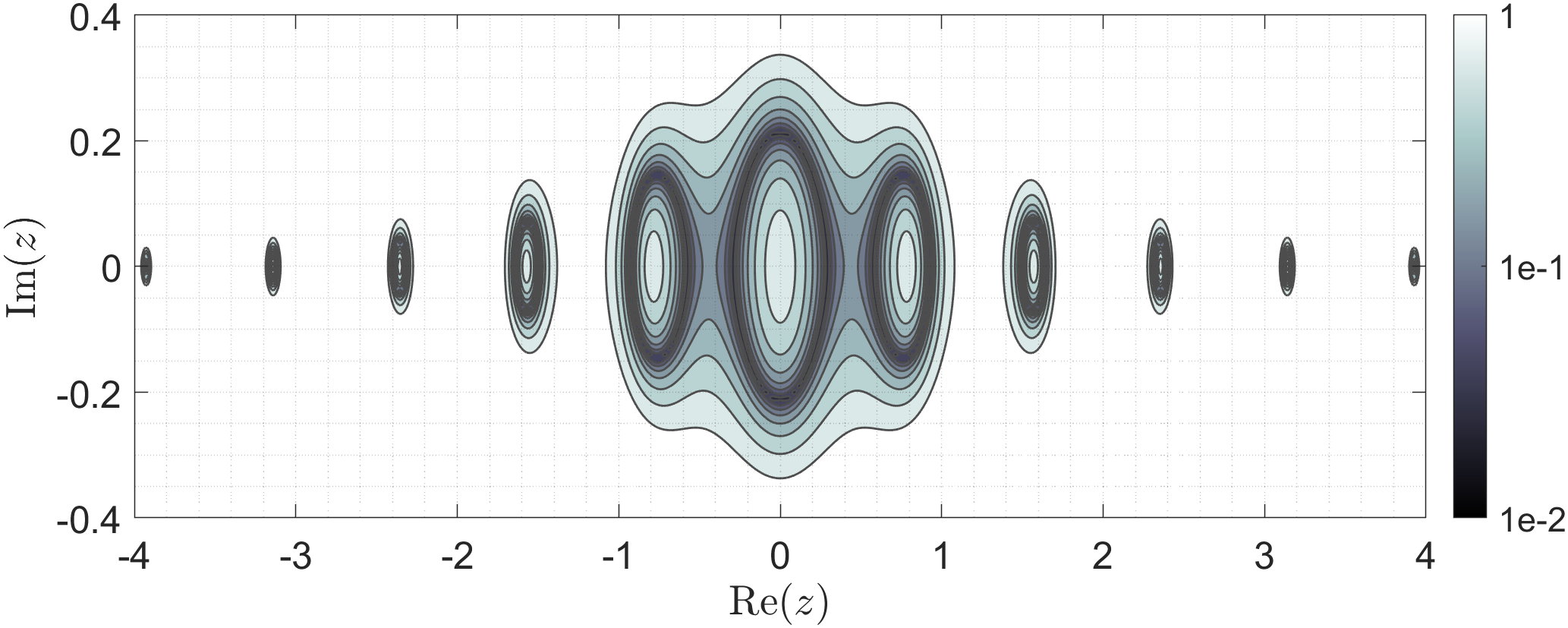}
\caption{Approximation of pseudospectra for the nonlinear shift example using \cref{alg1} and rectangular truncations to obtain a $\Sigma_1^A$ tower. The color corresponds to values of $\epsilon$ and we show contour levels that display the boundaries of $\spec_\epsilon(T)$ for a logarithmic grid of $\epsilon$ values. In contrast to the finite section method in \cref{fig:nonlinear_shift0}, this method converges.}
\label{fig:nonlinear_shift}
\end{figure}

\subsection{Klein--Gordon equation}

In quantum mechanics, the Klein--Gordon equation,
$$
\left(\left[\frac{\partial}{\partial t}-ieq\right]^2-\nabla^2 +m^2\right)\psi=0,
$$
describes the motion of a relativistic spinless particle of mass $m$ and charge $e$ in an electrostatic field with potential $q$. Here, $\psi(x,t)$ is a wavefunction with $x\in\mathbb{R}^d$. If we seek separable solutions $\psi(x,t)=e^{i\lambda t}f(x)$, we are led to $M(\lambda)f=0$ for the nonlinear pencil
$$
M(\lambda) = -\nabla^2 +m^2I- (eq-\lambda )^2I.
$$
More generally, on a Hilbert space $\mathcal{H}$, we can consider a nonnegative operator $H_0$ (instead of $-\nabla^2 +m^2I$) and a self-adjoint operator $V$ (instead of $eqI$), which we take to be bounded for simplicity, leading to the pencil
$$
T(\lambda) = H_0 - (V-\lambda I)^2.
$$
There is considerable interest in the spectral and scattering properties of such Klein--Gordon equations; see \cite{langer2006spectral,lundberg1973spectral,schiff1940existence,weder1978scattering,langer2008spectral} and the references therein for a very small sample. See also \cite{rosler2022computing} for SCI classifications of computing spectra in the case of $H_0=-\nabla^2 +m^2I$, which take advantage of the extra structure of the problem.

As a concrete example, we consider the Hilbert space $\mathcal{H}=l^2(\mathbb{Z})$,
$$
H_0=\begin{pmatrix}
\ddots& \ddots& \ddots& &&\\
 &\frac{3}{2}&2&\frac{1}{2} & & &\\
 &&\frac{1}{2} &|2|& \frac{3}{2}& &\\
 && & \frac{3}{2}&2& \frac{1}{2}&\\
&&&&\ddots&\ddots&\ddots\\
\end{pmatrix},\quad\text{and}\quad [Vx]_n=-5\exp(-|n|)x_n,\quad x\in l^2(\mathbb{Z}),n\in\mathbb{Z}.
$$
Here, $|2|$ signifies the $(0,0)$th entry of $H_0$, and the $\ddots$ in the subdiagonal and superdiagonal indicate a periodic repetition with a period of two. The essential spectrum of $H_0$ is $[0,1]\cup[3,4]$. Since $V$ is compact, the essential spectrum of $T$ is $[-2,-\sqrt{3}]\cup[-1,1]\cup[\sqrt{3},2]$. The rest of the spectrum consists of discrete eigenvalues.

\begin{figure}
\centering
\includegraphics[width=1\linewidth]{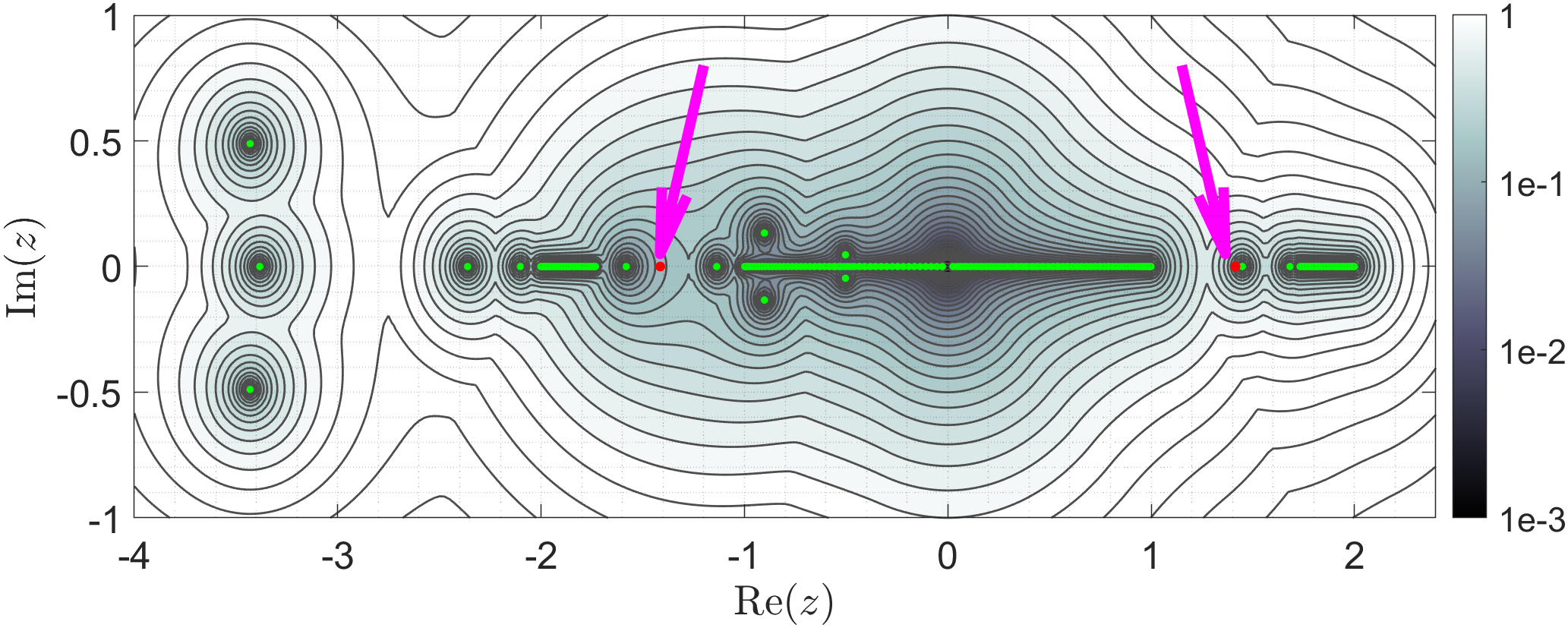}
\caption{Output of \cref{alg1} for the Klein--Gordon spectral problem. The green and red dots correspond to the eigenvalues of the quadratic eigenvalue problem resulting from truncation. The two eigenvalues shown in red are spurious and are highlighted by the arrows.}
\label{fig:KG1}
\end{figure}

\cref{fig:KG1} shows the output of \cref{alg1}, where we approximate the injection modulus using
$$
\gamma_{n}(z,T)=\min\left\{\sigma_{\mathrm{inf}}(\mathcal{P}^{\mathbb{Z}}_{n+1}T(z){\mathcal{P}^{\mathbb{Z}}_n}^*),\sigma_{\mathrm{inf}}(\mathcal{P}^{\mathbb{Z}}_{n+1}T(z)^*{\mathcal{P}^{\mathbb{Z}}_n}^*)\right\}.
$$
with $n=100$. We have also shown the eigenvalues of the quadratic eigenvalue problem $\mathcal{P}^{\mathbb{Z}}_{n}T(z){\mathcal{P}^{\mathbb{Z}}_n}^*$ as green and red dots. The red dots, highlighted with arrows, show spectral pollution in gaps of the essential spectrum. This demonstrates that naive truncation methods can fail and give misleading results, even for simple quadratic eigenvalue problems. We can also adapt \cref{alg2} to compute the discrete eigenvalues by searching for local minima of $\gamma_{n}(z,T)$. \cref{fig:KG2} shows the convergence of this approach for three representative eigenvalues, where we have shown $\gamma_{n}(z,T)$ which provides an \textit{upper bound} on $\|T(z)^{-1}\|^{-1}$. The convergence in this case is exponential in the truncation parameter $n$.

\begin{figure}
\centering
\includegraphics[width=0.5\linewidth]{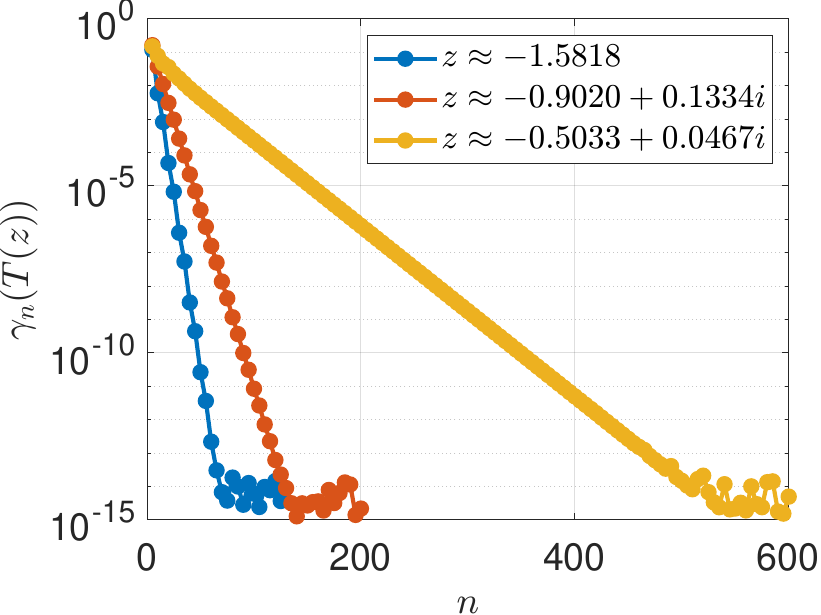}
\caption{Convergence of local minima of $\gamma_{n}(z,T)$ (essentially an adaptation of \cref{alg2}) for three representative eigenvalues in the discrete spectrum. The legend shows the (rounded) approximate location of the eigenvalues. The plateauing of the curves around $10^{-15}$ is due to reaching machine precision.}
\label{fig:KG2}
\end{figure}

\subsection{Wave equation with acoustic boundary condition}

We now consider the following equation:
\begin{equation}\label{inf_acoustic_wave}
\frac{d^2 p}{dx^2} +z^2 p=0 \quad x\in(0,\infty),\quad -\frac{d p}{dx}(0)+izp(0)=0.
\end{equation}
Here, $p$ represents acoustic pressure with an impedance boundary condition at $x=0$, and $z$ denotes the frequency, treated as the spectral parameter. Wave equations equipped with time-dependent acoustic boundary conditions are discussed in \cite{mugnolo2006abstract}. \Cref{inf_acoustic_wave} leads to a pencil $T(z)$ on $L^2((0,\infty))$ with a domain dependent on $z$. It can be shown that
$$
\spec(T)=\{z\in\mathbb{C}:\mathrm{Im}(z)\geq 0\}.
$$
A natural first step toward addressing such problems is truncating the domain to an interval $[0,L]$ for sufficiently large $L$, imposing the Dirichlet boundary condition $p(L)=0$. Introducing the scaled spectral parameter $\lambda = L z$ results in a spectrum independent of $L$. The resulting problem is typically discretized using finite element methods (FEMs) \cite{harari1996recent}, leading to a finite-dimensional quadratic eigenvalue problem \cite{chaitin2006analysis}. This domain-truncated and FEM-discretized problem also appears in the benchmark NLEVP collection \cite{betcke2013nlevp}.

\cref{fig:acoustic_wave_quad_evals} (left) displays eigenvalues of the truncated and discretized problem for four discretization sizes: $n=10$, $100$, $500$, and $1000$. These eigenvalues are computed using the \texttt{polyeig} command in MATLAB. The spectrum of the truncated problem (prior to discretization via FEM) is empty; therefore, all eigenvalues shown in \cref{fig:acoustic_wave_quad_evals} may be considered spurious. However, they do not constitute spectral pollution (with respect to the truncated problem), as they tend toward infinity—albeit logarithmically slowly—with the minimal absolute value of eigenvalues growing as $\mathcal{O}(\log(n))$, illustrated in \cref{fig:acoustic_wave_quad_evals} (right).

\begin{figure}
\centering
\includegraphics[width=0.495\linewidth]{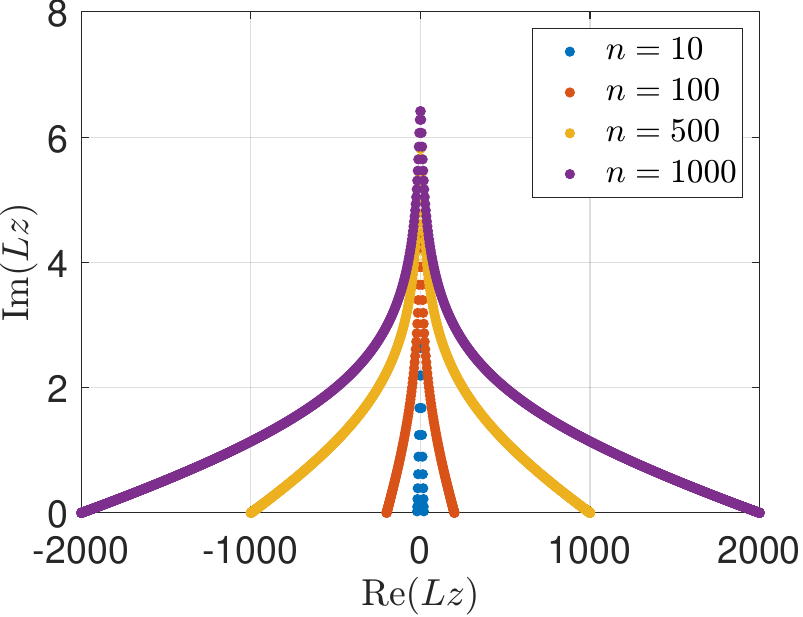}\hfill
\includegraphics[width=0.46\linewidth]{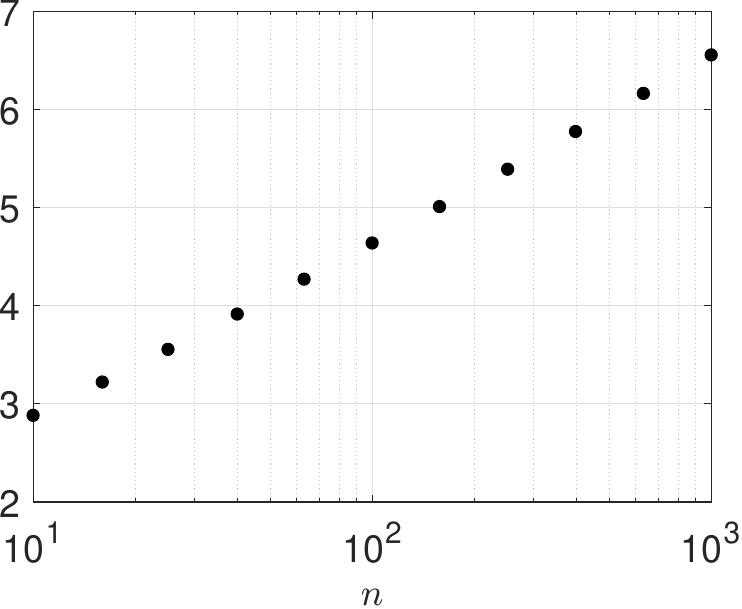}
\caption{Left: Eigenvalues of the quadratic eigenvalue problem after domain truncation followed by FEM discretization. Right: The minimum of the absolute value of the eigenvalues for different discretization sizes $n$.}
\label{fig:acoustic_wave_quad_evals}
\end{figure}

To apply \cref{alg1,alg2} and compute the spectral properties of $T$, we first consider the Laguerre functions
$$
\phi_n(x) = L_n(x)e^{-x/2},\quad L_n(x)=\frac{1}{n!}\left(\frac{d}{dx}-1\right)^n x^n\quad n=0,1,2,\ldots.
$$
The functions $\{\phi_n\}_{n=0}^\infty$ form an orthonormal basis of $L^2((0,\infty))$. However, they do not lie in the domain of $T(z)$. Therefore, we introduce the modified basis functions
$$
\hat{\phi}_n(x)=\phi_n(x)+\alpha_n \phi_0(x),\quad \alpha_n=-(2iz+2n+1)/(2iz+1),
\quad n\in\mathbb{N}.
$$
Since $\phi_n(0)=1$ and $\phi_n'(0)=-(n+1/2)$, these modified functions satisfy the given boundary conditions. These functions, however, are not orthogonal, and it is more convenient to use an orthogonal basis when computing injection moduli. Thus, we apply the Gram–Schmidt process (numerically implemented via a QR decomposition for numerical stability) to obtain an orthonormal basis $\{f_n\}_{n=1}^\infty$. This resulting orthonormal basis is dependent on $z$, illustrating a generalisation of our convergence results to cases where the bases used to represent the nonlinear pencil (in $\mathcal{H}_1$ and $\mathcal{H}_2$) may depend on the spectral parameter $z$. The convergence results in \cref{sec:basic_convergence} continue to hold under this relaxed assumption. Moreover, the matrix representation of $T$ in the basis $\{f_n\}_{n=1}^\infty$ has zero entries below the first subdiagonal (and similarly for an analogous discretisation of the adjoint). Hence, we can use rectangular truncations (setting $n_2=n_1+1$ in \cref{alg1}) to achieve $\Sigma_1^A$ convergence to pseudospectra and $\Pi_2^A$ convergence to the spectrum.

\Cref{fig:acoustic_wave_pseudospectra} illustrates the convergence of \cref{alg1} to pseudospectra using $N$ basis functions. \Cref{alg1} can also be employed to compute pseudoeigenfunctions. Given an approximation $\gamma_{N+1,N}(z,T)$ of $\|T(z)^{-1}\|^{-1}$, we compute right-singular vectors associated with the smallest singular value during the computation of injection moduli. For instance, we compute $u \in \mathrm{span}\{f_1,\ldots,f_N\}$ satisfying (up to controllable numerical errors)
$$
\|T(z)u\|/\|u\|=\sigma_{\mathrm{inf}}(T\mathcal{P}_N^*),
$$
where $\mathcal{P}_N$ denotes the orthogonal projection onto $\mathrm{span}\{f_1,\ldots,f_N\}$. \Cref{fig:acoustic_wave_efun} shows examples for $z=2\pi+i$, where we have normalised pointwise by the true eigenfunction $e^{izx}$. The convergence of the approximations are evident. Similar behaviour occurs for real $z$, where there are no normalisable eigenfunctions.

\begin{figure}
\centering
\includegraphics[width=0.32\linewidth]{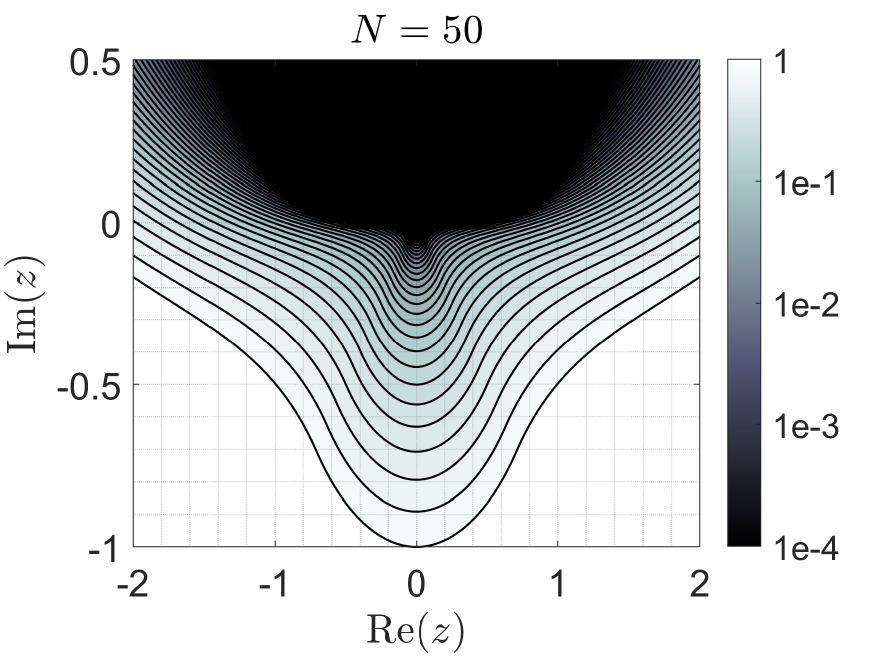}
\includegraphics[width=0.32\linewidth]{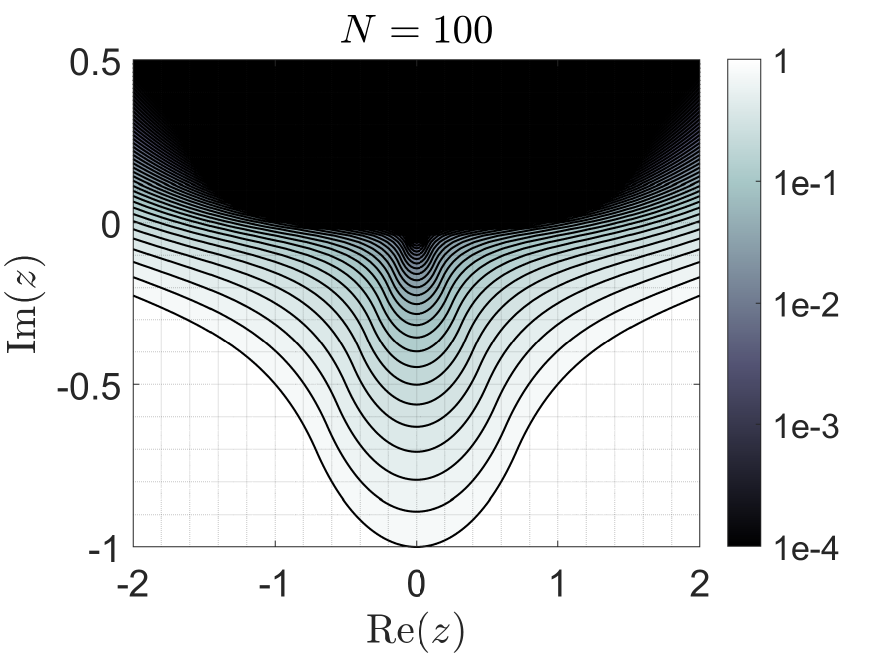}
\includegraphics[width=0.32\linewidth]{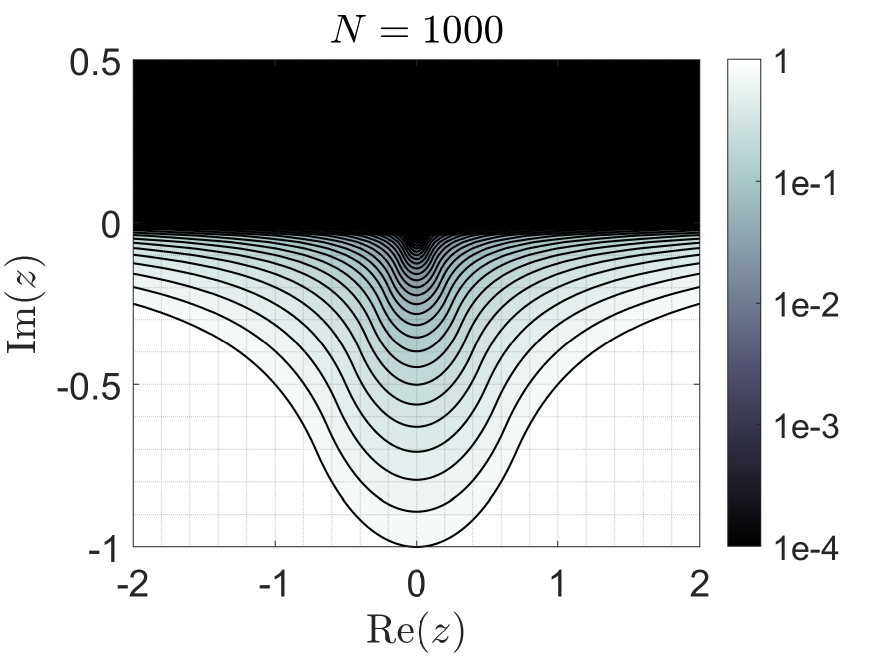}
\caption{Convergence of \cref{alg1} (as we move from left to right) for the wave equation example with acoustic boundary condition. In this case, the spectrum is $\{z\in\mathbb{C}:\mathrm{Im}(z)\geq 0\}$.}
\label{fig:acoustic_wave_pseudospectra}
\end{figure}

\begin{figure}
\centering
\includegraphics[width=0.32\linewidth]{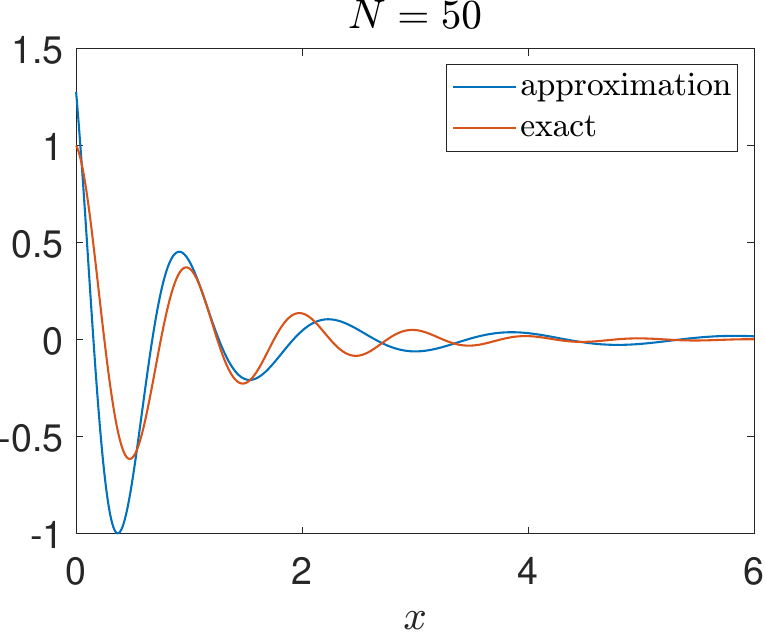}
\includegraphics[width=0.32\linewidth]{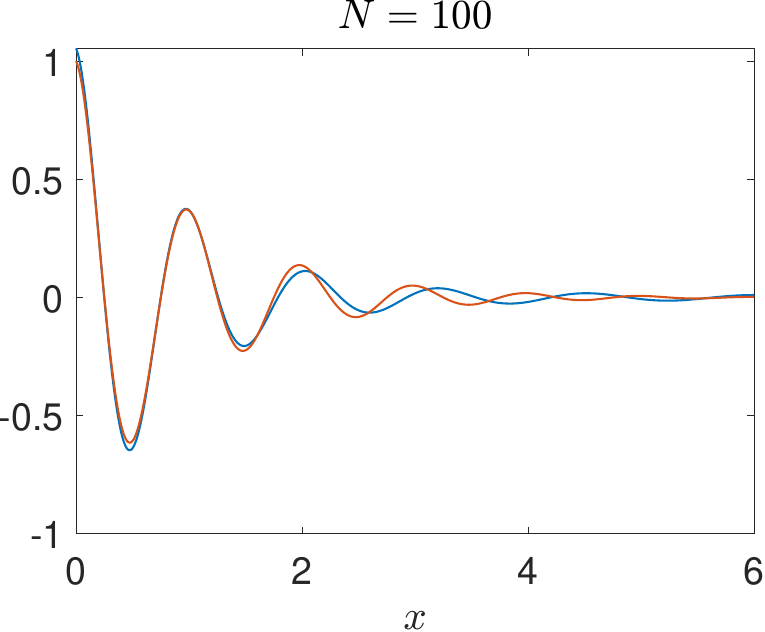}
\includegraphics[width=0.32\linewidth]{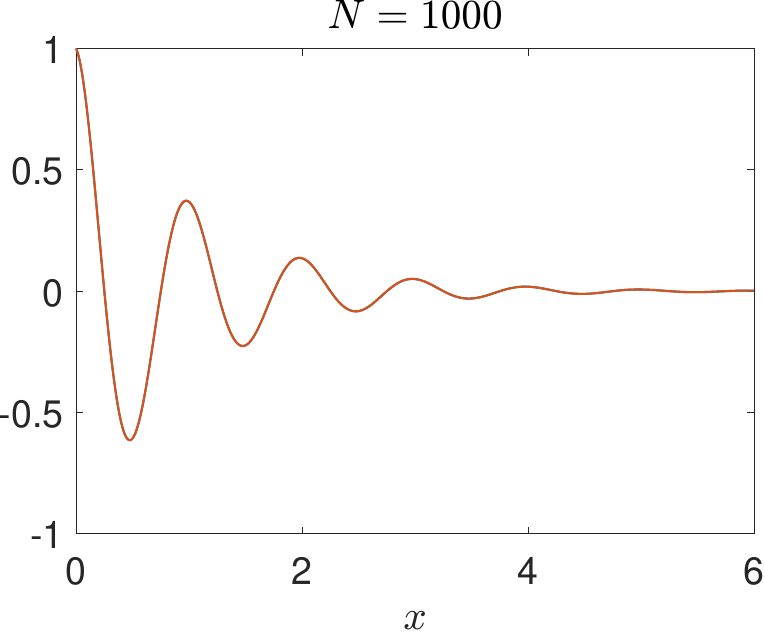}
\caption{Convergence of pseudoeigenfunctions for $z=2\pi+i$.}
\label{fig:acoustic_wave_efun}
\end{figure}

\subsection{Time-fractional viscoelastic beam equation}

Fractional differential equations are increasingly employed to model complex physical and biological systems \cite{hilfer2000applications,mainardi2010fractional,sabatier2007advances,magin2010fractional,metzler2000random}. However, their nonlocal nature introduces significant numerical challenges. Traditional methods—such as finite differences, finite elements, and convolution quadrature (see the overviews in \cite{baleanu2012fractional,li2015numerical})—often struggle with accuracy and efficiency, particularly over large domains. 
The Laplace transform provides a connection between the solutions of certain evolution equations and the spectral properties of operators \cite{arendt2001cauchy}. Contour methods based on Laplace transform inversion offer high accuracy and parallelisability for time-fractional differential equations, while avoiding the memory constraints and singularity resolution (as $t \downarrow 0$) required by time-stepping schemes \cite{trefethen2014exponentially,weideman2006optimizing,pang2016fast,mclean2010maximum,mclean2010numerical}. These methods also lend themselves to the computation of solutions to evolution equations on infinite-dimensional spaces \cite{colbrook2022computing,colbrook2022contour}.

As an example, we consider the following viscoelastic beam equation:
\begin{equation}
\begin{split}
    \pderiv{^{2}y}{t^{2}}+\pderiv{^{2}}{x^{2}}\left(\left[a(x)+b(x)\mathcal{D}^{\nu}_{C,t} \right]\pderiv{^{2}y}{x^{2}}\right)=F(x,t),\quad (x,t)\in[-1,1]\times[0,\infty),\\
		y(\pm 1,t)=\pderiv{y}{x}(\pm 1,t)=0.
		\end{split}
    \label{eq:gov}
\end{equation}
Here, $y(x,t)$ is the transverse displacement of the beam, $F(x,t)$ a transverse loading, $a$ and $b$ are (for simplicity) smooth positive functions, and $\mathcal{D}^{\nu}_{C,t}$ is Caputo's fractional derivative \cite{caputo1967linear} with $\nu\in(0,2)$. The boundary conditions correspond to the beam being clamped at both endpoints. \cref{eq:gov} arises in the modeling of modern materials embedded with polymer structures or biomaterials, which exhibit both elastic and viscous behavior. These effects can be experimentally measured and fitted to constitutive stress-strain relationships \cite{Pritz}. Incorporating time-fractional derivatives in the constitutive law is a popular approach for accurately capturing experimental data with a small number of parameters \cite{KV1,KV_AIAA}.

The following formal steps are analysed in \cite{colbrook2022contour}, and we work in the Hilbert space $L^2([-1,1])$. Taking Laplace transforms $\hat{\cdot}$ of \cref{eq:gov}, we arrive at
$$
\left[z^2 I +\pderiv{^{2}}{x^{2}}\left(\left[a(x)+z^\nu b(x) \right]\pderiv{^{2}}{x^{2}}\right)\right]\hat{y}(x,z)=\hat{G}(x,z),
$$
where $\hat{G}$ can be written in terms of the Laplace transform of $F$ and initial conditions for $y$. Let $S$ denote the initial fractional pencil
$$
S(z)=z^2 I +\pderiv{^{2}}{x^{2}}\left(\left[a(x)+z^\nu b(x) \right]\pderiv{^{2}}{x^{2}}\right),\quad \mathcal{D}(S)=\{u\in H^4([-1,1]):u(\pm1)=u'(\pm1)=0 \}.
$$
We then take the closure $T(z)$ of $S(z)$ and invert the Laplace transform to obtain
\begin{equation}
\label{eq:inv_lap}
y(x,t)=\frac{1}{2\pi i}\int_{\omega -i\infty}^{\omega+i\infty}e^{zt}T(z)^{-1}\hat{G}(x,z)\, \mathrm{d}z,
\end{equation}
where $\omega \in \mathbb{R}$ is chosen so that the singularities of $T(z)^{-1} \hat{G}(x,z)$ lie to the left of the contour.
The key to making this a numerically competitive solution method is to deform the contour of integration so that a rapidly convergent quadrature rule may be applied. To choose a suitable contour, one must first bound the region containing the singularities of $\hat{y}(x,z)$—that is, the spectrum of $T$—and study $\|T(z)^{-1}\|$, i.e., the pseudospectrum of $T$. We do not discuss the procedure for selecting optimal contours, which can be found in \cite{colbrook2022computing,colbrook2022contour}. Overestimating $\|T(z)^{-1}\|$ can significantly degrade the convergence rate of the solution method. Hence, obtaining accurate estimates is vital. One can even use computed pseudospectra to devise contours \cite{guglielmi2020pseudospectral}.

In \cite[Theorem 3.4]{colbrook2022contour}, $\|T(z)^{-1}\|$ is bounded using a quasi-linearisation of the operator pencil (to handle the quadratic term) and a generalisation of the numerical range. Here, we take a direct approach. The advantage of this is that it yields sharper estimates on $\|T(z)^{-1}\|$.

To apply \cref{alg1,alg2} and compute the spectral properties of $T$, we first consider the Legendre polynomials
$$
P_n(x) = \frac{1}{2^nn!}\frac{d^n}{dx^n}(x^2-1)^n,\quad n=0,1,2,\ldots.
$$
These functions do not lie in the domain of $T(z)$. Therefore, we introduce the modified basis functions
$
(x^2-1)^2P_n(x)
$
to satisfy the given boundary conditions. We then apply the Gram–Schmidt process (numerically implemented via a QR decomposition for numerical stability) to obtain an orthonormal basis. The matrix representation of $T$ in this basis has a finite lower bandwidth. Hence, we can use rectangular truncations to achieve $\Sigma_1^A$ convergence to pseudospectra and $\Pi_2^A$ convergence to the spectrum.

We consider the choices
$$
a(x)= 1+\frac{\cos(x)}{2},\quad b(x) = \frac{3}{2} + \tanh(10x).
$$
\Cref{fig:viscoelastic1,fig:viscoelastic2,fig:viscoelastic3} show the pseudospectra computed using 100 basis functions for $\nu = 0.5$, $1$, and $1.5$, respectively. We also compare these with the numerical-range-type bounds from \cite[Theorem 3.4]{colbrook2022contour}. In all cases, computing the pseudospectra using \cref{alg1} yields significantly sharper estimates (as expected), and can therefore aid in the design of efficient numerical solvers.

\begin{figure}
\centering
\includegraphics[width=0.49\linewidth]{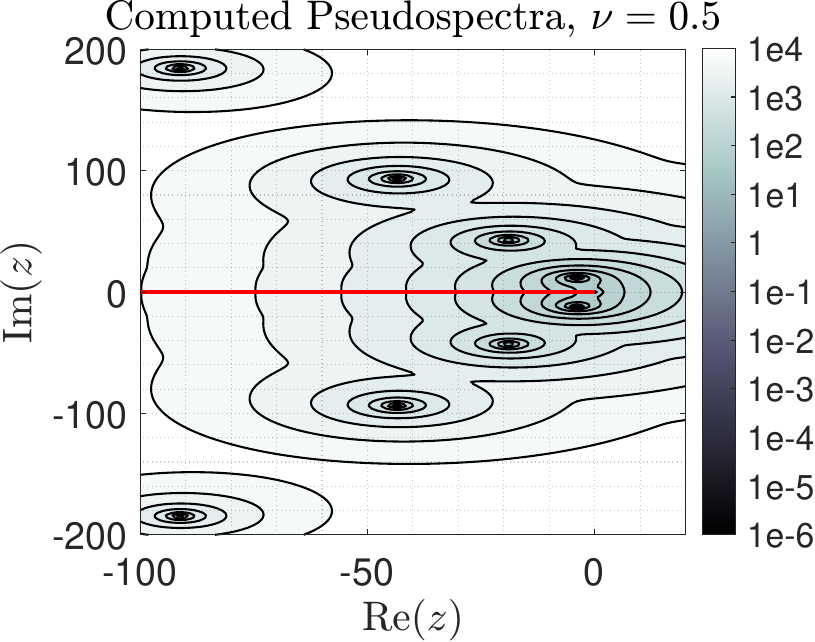}\hfill
\includegraphics[width=0.49\linewidth]{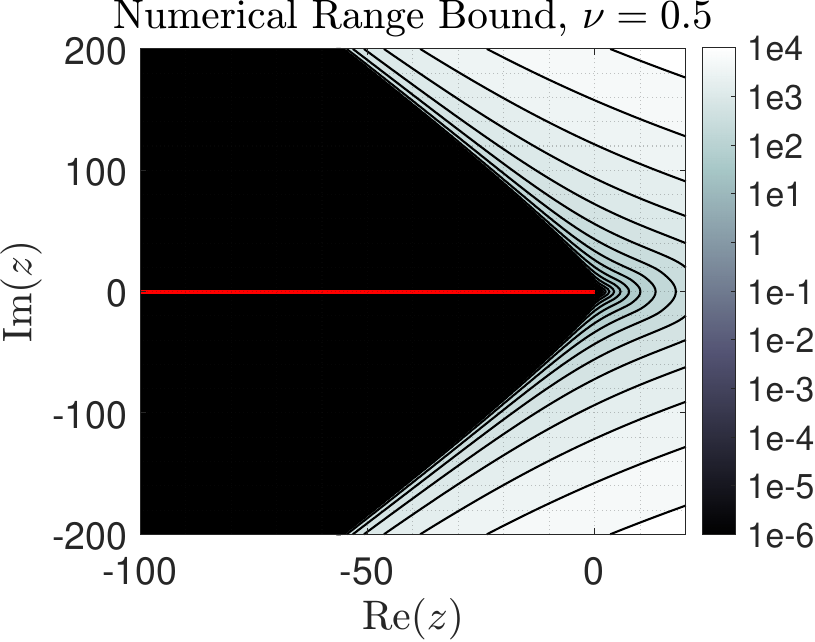}
\caption{Comparison of pseudospectra computed using \cref{alg1} (left) and the numerical range lower bounds on $\|T(z)^{-1}\|^{-1}$ (right). The branch cut for the non-integer power of $z$ is shown as a red line.}
\label{fig:viscoelastic1}
\end{figure}

\begin{figure}[t]
\centering
\includegraphics[width=0.49\linewidth]{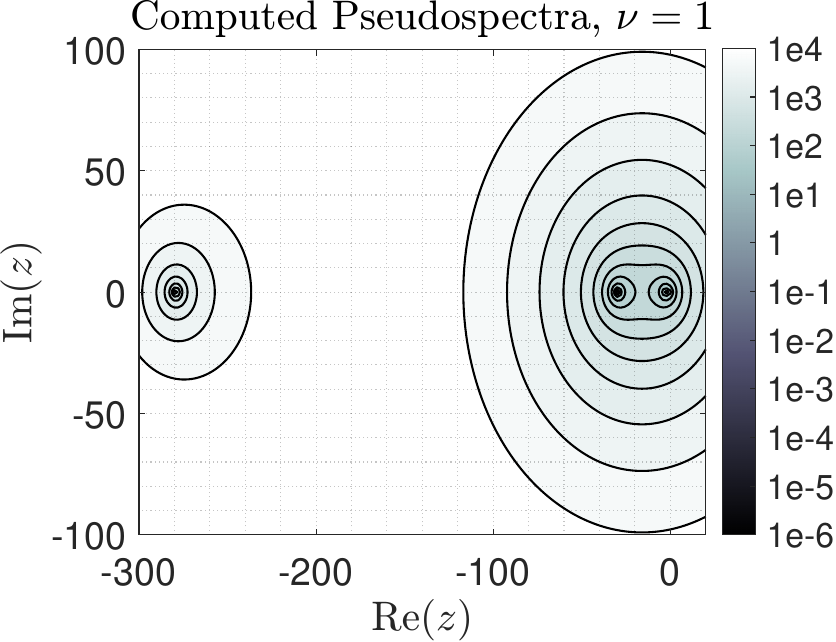}\hfill
\includegraphics[width=0.49\linewidth]{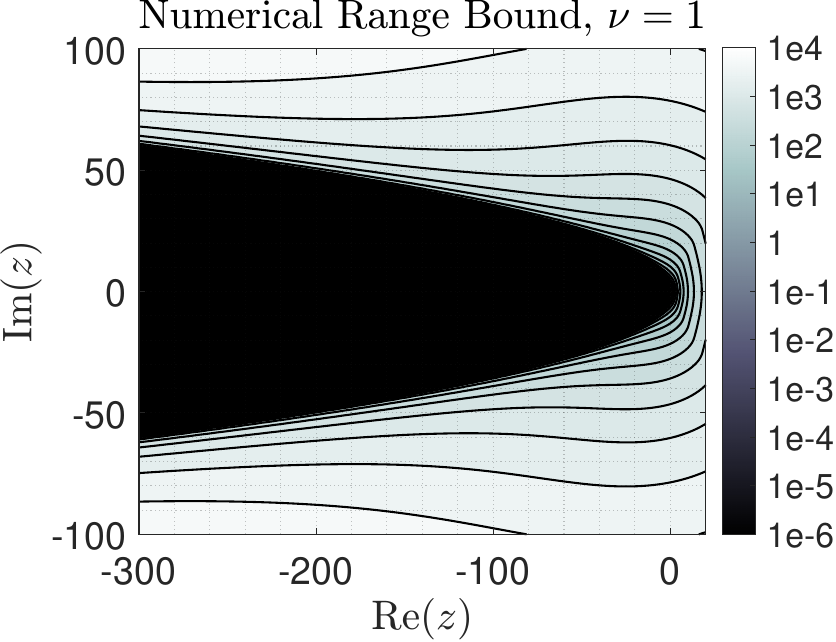}
\caption{Similar to \cref{fig:viscoelastic1}, but for $\nu=1$.}
\label{fig:viscoelastic2}
\end{figure}

\begin{figure}[t]
\centering
\includegraphics[width=0.49\linewidth]{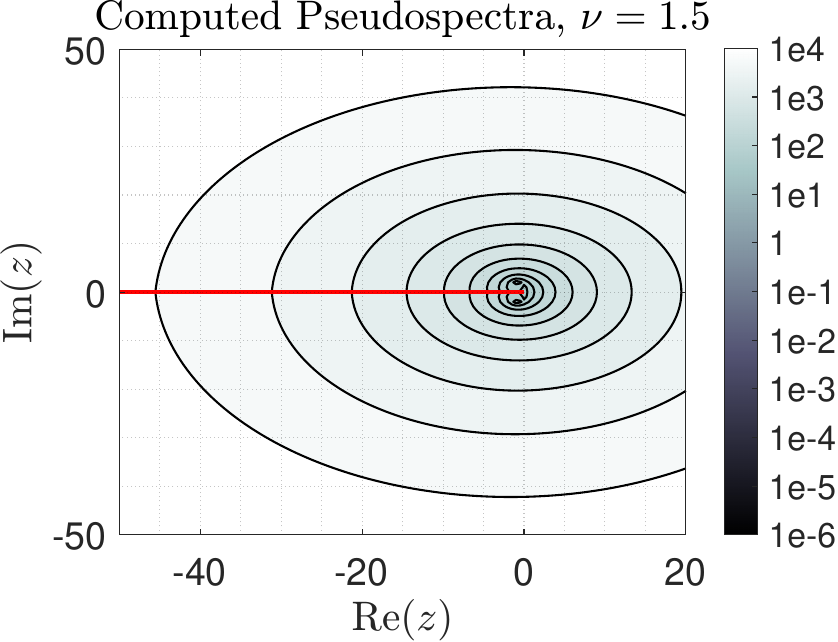}\hfill
\includegraphics[width=0.49\linewidth]{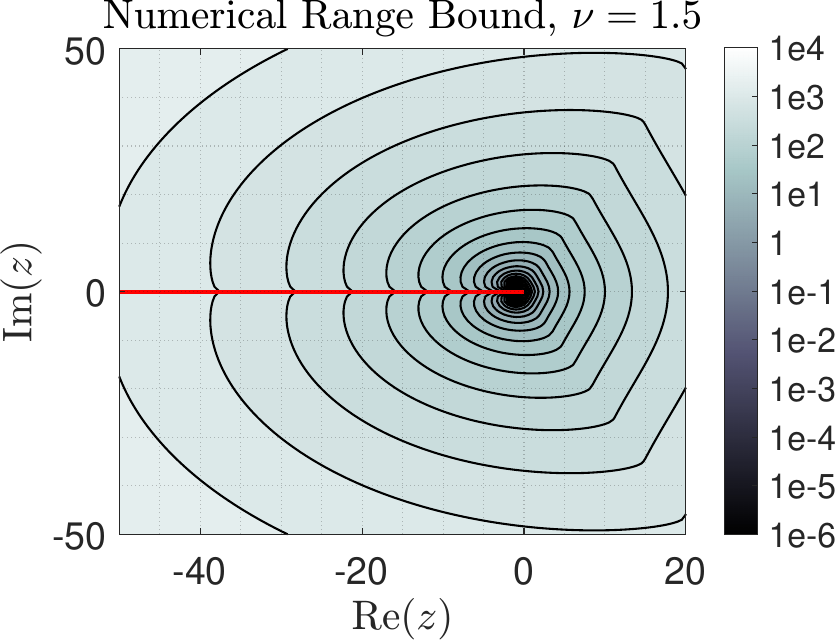}
\caption{Similar to \cref{fig:viscoelastic1}, but for $\nu=1.5$. In this case, the spectrum is bounded.}
\label{fig:viscoelastic3}
\end{figure}

\subsection{Diffusive Lotka--Volterra predator-prey system with delay}

We now consider the following predator-prey model with delay \cite{faria2001stability}:
\begin{gather*}
\frac{\partial u}{\partial t}(x,t)=d_1\tau \frac{\partial^2 u}{\partial^2 x}(x,t)+\tau u(x,t)\left[r_1-v(x,t-r)\right],\\
\frac{\partial v}{\partial t}(x,t)=d_2\tau \frac{\partial^2 v}{\partial^2 x}(x,t)+\tau v(x,t)\left[-r_2+u(x,t-1)\right],\\
u(-1,t)=2,\quad v(-1,t)=1,\quad u(1,t)=v(1,t)=0.
\end{gather*}
Here, $u(\cdot,t)$ and $v(\cdot,t)$ represent prey and predator population densities, respectively, on the domain $[-1,1]$. A discussion of the biological interpretation of these (and other possible) boundary conditions is given in \cite{fagan1999habitat}. Throughout this example, we take $d_1=0.15$, $d_2=0.5$, $\tau=0.1$, $r=\sqrt{2}$ and $r_1=1$. Following \cite{faria2001stability}, we consider steady-state (time-independent) solutions $(u^*,v^*)$. \cref{fig:prey1} shows examples of steady-state solutions, computed using a Chebyshev spectral method for $r_2=0.5$ and $r_2=100$.

\begin{figure}[t]
\centering
\includegraphics[width=0.49\linewidth]{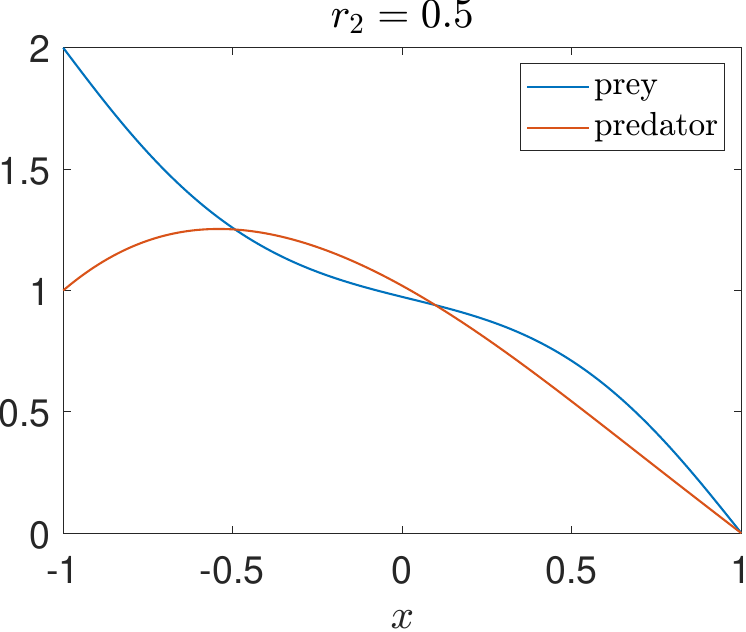}\hfill
\includegraphics[width=0.49\linewidth]{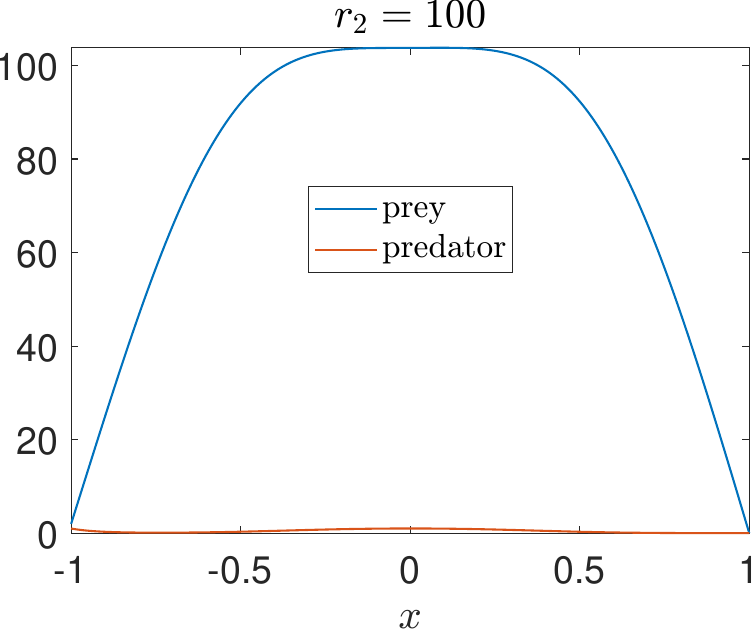}
\caption{Steady-state solutions of the diffusive Lotka--Volterra predator-prey system with delay.}
\label{fig:prey1}
\end{figure}

We then linearise the equations about these solutions and consider
$$
u(x,t)=u^*(x)+e^{\lambda t}\hat{u}(x),\quad v(x,t)=v^*(x)+e^{\lambda t}\hat{v}(x).
$$
This leads to the equations
\begin{gather*}
\begin{pmatrix}
\lambda-d_1\tau \frac{d^2}{d^2 x} -\tau r_1 +\tau v^* & e^{-r\lambda}\tau u^*\\
- e^{-\lambda}\tau v^*& \lambda-d_2\tau \frac{d^2}{d^2 x}+\tau r_2-\tau u^*
\end{pmatrix}
\begin{pmatrix}
\hat{u}\\
\hat{v}
\end{pmatrix}=0,\quad\hat{u}(\pm 1)=\hat{v}(\pm 1)=0,
\end{gather*}
which induce a corresponding operator pencil on $L^2([-1,1])\times L^2([-1,1])$. To apply \cref{alg1,alg2}, we consider the functions $(x^2-1)P_n(x)$, where, as above, $P_n$ denotes the $n$th Legendre polynomial. These functions satisfy the boundary conditions, and we apply the Gram--Schmidt process (numerically implemented via a QR decomposition for numerical stability) to obtain an orthonormal basis of $L^2([-1,1])$. We then expand using $500$ basis functions for each of $\hat{u}$ and $\hat{v}$. \cref{fig:prey2} shows the resulting pseudospectra. The steady-state solutions are stable for $r_2=0.5$ but unstable for $r_2=100$. However, even for $r_2=0.5$, large regions of the pseudospectra protrude into the right-half plane.

\begin{figure}[t]
\centering
\includegraphics[width=0.49\linewidth]{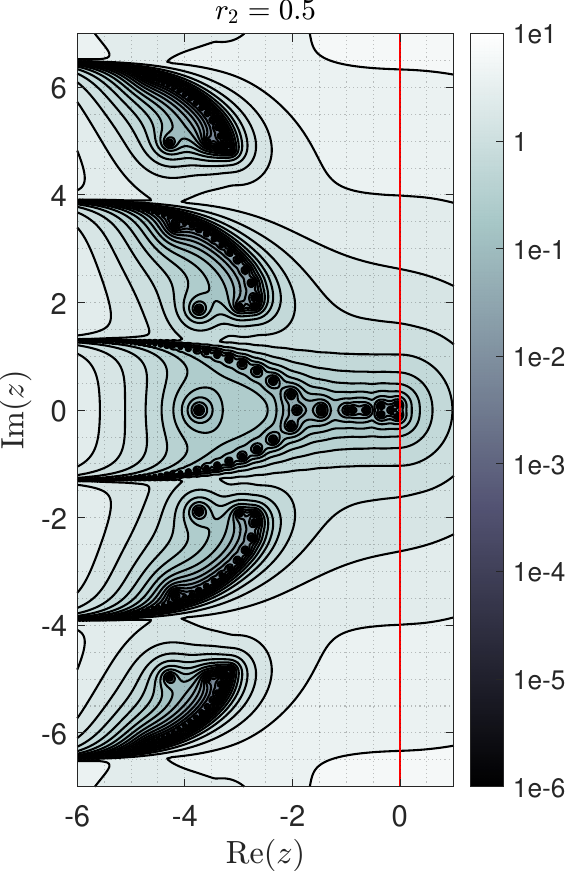}\hfill
\includegraphics[width=0.49\linewidth]{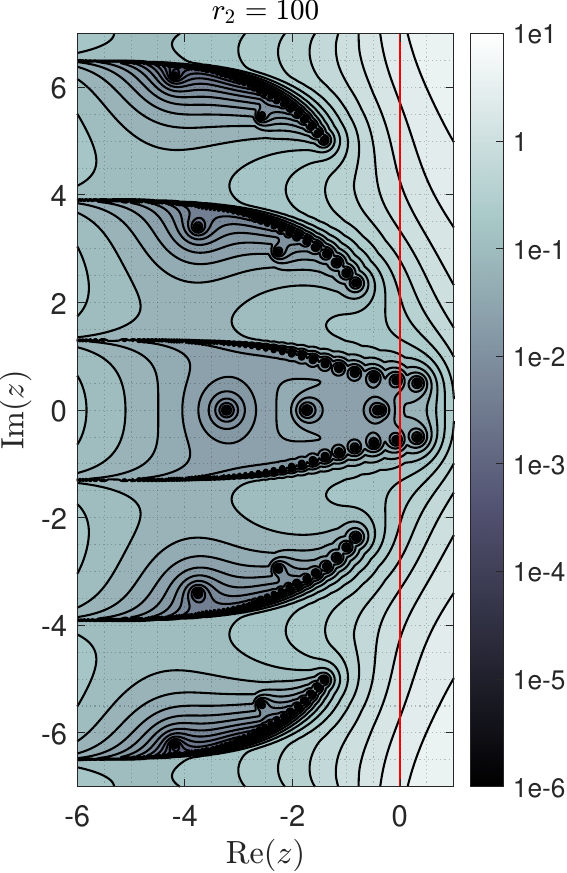}
\caption{Pseudospectra of the linearised predator-prey system, computed using \cref{alg1}. The red lines show the imaginary axis.}
\label{fig:prey2}
\end{figure}

\section*{Acknowledgements}
We would like to thank the London Mathematical Society for providing funds under their Research in Pairs Scheme for in-person collaboration in the Spring of 2025.

\small
\bibliographystyle{abbrv}

\bibliography{spectra_bib_final,bib_new}

\end{document}